\documentclass[11pt]{amsart}

\usepackage[T1]{fontenc}
\usepackage[a4paper,top=4cm,bottom=4cm,inner=3cm,outer=3cm]{geometry}
\usepackage[pdftex]{graphicx,graphics} 
\usepackage{amsmath,amssymb,amsfonts,color,hyperref} 
 \usepackage{epsfig,array}

\usepackage{graphicx,psfrag,mathrsfs}
\usepackage{tikz}  
\usepackage{slashbox}
\usetikzlibrary{shapes.geometric, arrows,snakes}

\hypersetup{
    linktoc=page,
   linkcolor=red,          % color of internal links
  citecolor=blue,        % color of links to bibliography
    filecolor=blue,      % color of file links
   urlcolor=cyan,
    colorlinks=true           % color of external links
}

\usepackage[refpage]{nomencl}

\newcommand{\lan}{{\mathcal L}}
\newcommand\var{\text{var}}

\newcommand\vl{\underline{v}}
\newcommand\wl{\underline{w}}
\newcommand\dl{{\underline{\dim}_{\rm loc}}}
\newcommand\du{{\overline{\dim}_{\rm loc}}}

\newcommand\El{\underline{E}}
\newcommand\supp{\mathrm{supp}}

\RequirePackage{yhmath} % dÂfinit \widering, \wideparen etc...
\renewcommand\widering[1]{\ring{#1}}%{\wideparen {#1}}
\newcommand\R{{\Bbb R}}
\newcommand\N {{\Bbb N}}

\newtheorem{theorem}{Theorem}[section]

\newtheorem{lemma}[theorem]{Lemma}
\newtheorem{proposition}{Proposition}

\theoremstyle{definition}
\newtheorem{definition}[theorem]{Definition}
\newtheorem{remark}{Remark}

\newtheorem{fact}{Fact}

\title[Multifractal of random weak Gibbs measures] %Use the shortened version of the full title
{Multifractal analysis of random weak Gibbs measures}

\author[Zhihui YUAN]{}

%%Primary: 37H99; Secondary: 53C35.  37C45Dimension theory of dynamical systems,37Hxx Random dynamical systems

\email{yzhh@hust.edu.cn}
\begin{document}
\maketitle

% Enter the first author's name and address:
\centerline{\scshape Zhihui Yuan}%%$^*$
\medskip

\begin{abstract}
	We describe the multifractal nature of random weak Gibbs measures on some class of attractors associated with $C^1$ random dynamics semi-conjugate  to a random subshift of finite type. This includes the validity of the multifractal formalism, the calculation of Hausdorff and packing dimensions of the so-called level sets of divergent points,  and a $0$-$\infty$ law for the Hausdorff and packing measures of the level sets of the local dimension.
\end{abstract}

%The title of your section 1
\section{Introduction}

Weak Gibbs measures are conformal probability measures obtained  as eigenvectors of Ruelle-Perron-Frobenius operators associated with continuous potentials  on topological dynamical systems.  When the system $(X,f)$ has nice enough geometric properties, for instance in the case of a conformal repeller, these measures  provide natural, and now standard examples of measures obeying the multifractal formalism: their Hausdorff spectrum  and  $L^q$-spectrum form a Legendre pair.

Specifically, for such a measure $\mu$ on $(X,f)$, the (lower) $L^q$-spectrum $\tau_{\mu}:\mathbb{R}\rightarrow\mathbb{R}\cup\{-\infty\}$ is defined by
\begin{equation}%%\label{Lq}
\tau_{\mu}(q)=\liminf_{r\rightarrow 0}\frac{\log \sup\{\sum_{i}(\mu(B_{i}))^q\}}{\log (r)},
\end{equation}
where the supremum is taken over all families of disjoint closed balls $B_i$ of radius $r$ with centers in $\text{supp}(\mu)$; the Hausdorff spectrum of $\mu$ is defined by
$$
d\in\mathbb{R} \mapsto \dim_H E(\mu,d),
$$
where $\dim_H$ stands for the Hausdorff dimension,
$$E(\mu,d)=\left \{x\in \mathrm{supp}(\mu):\dim_{\rm loc}(\mu,x)=d\right\},$$
with
$$
\displaystyle \dim_{\rm loc}(\mu,x)=\lim_{r\to 0^+}\frac{\log(\mu(B(x,r)))}{\log(r)},
$$
and we have the duality relation
$$
\dim_H E(\mu,d)=\tau_\mu^*(d):=\inf_{q\in\mathbb{R}}\{dq-\tau_\mu(q)\}, \quad \forall d\in \mathbb{R},
$$
a negative dimension  meaning that the set is empty. In fact, due to the super and submultiplicativity properties associated with $\mu$, the same equality holds if we replace the limit by a $\liminf$ or a $\limsup$ in the definition of the local dimension.

The rigorous study of these measures started with the Gibbs measures case, which  corresponds to H\"older continuous potentials, or continuous potentials  possessing the so-called bounded distorsions property, and in particular on the so-called "cookie-cutter" Cantor sets associated with a $C^{1+\alpha}$ expanding map $f$ on the line \cite{CLP1,Rand} (see \cite{PW1} for an extended discussion of dimension theory and multifractal analysis for hyperbolic conformal dynamical systems). This followed seminal works by physicists of turbulence and statistical mechanics pointing the accuracy of multifractals to statistically and geometrically describe the local behavior of functions and measures \cite{FrischParisi,Halsey}. In the case of Gibbs measures, the $L^q$-spectrum of the Gibbs measure is differentiable, and analytic if the potential $\phi$ is H\"older continuous; it is the unique solution $t$ of the equation $P(q\phi+t\log \|Df\|)=0$, where $P(\cdot)$ stands for the topological pressure. The general case of continuous potentials was solved later in \cite{Olivier,FFW,Kes,FLW}, with the same formula for the $L^q$-spectrum. These progress then led to the multifractal analysis of  Bernoulli convolutions associated with Pisot numbers \cite{FO,Feng3}. Thermodynamic formalism and large deviations are central tool in these studies. 

In the context of random dynamical systems, the multifractal analysis of random Gibbs measures (to be defined below) associated with random H\"older continuous potentials on attractors of random $C^{1+\alpha}$ expanding (or expanding in the mean) random conformal dynamics encoded by random subshifs of finite type has been studied in \cite{Kifer2}, \cite{FS} and  \cite{MSU}.  These works, as well as the dimension theory of attractors of random dynamics \cite{BG1,Kifer2,Kifer3,MSU}, are based on the thermodynamic formalism for random transforms \cite{Kifer1,Bog,BG1,Gundlach,KK,DG1,KL,DKS1,MSU,DKS2}. The multifractal analysis of random weak Gibbs measures is also implicitly considered in \cite{FS} (which deals with the multifractal analysis of Birkhoff averages), but the fibers are deterministic, and the techniques developed there seems difficult to adapt in a simple way in the case of random subshifts.

In this paper we consider, on a base probability space $(\Omega,\mathcal F,\mathbb{P},\sigma)$, random weak Gibbs measures on some class of attractors included in $[0,1]$ and associated with $C^1$ random dynamics semi-conjugate (up to countably many points), or conjugate, to a random subshift of finite type. We provide a study of the multifractal nature of these measures, including the validity of the multifractal formalism, the calculation of Hausdorff and packing dimensions of the so-called level sets of divergent points,  and a $0$-$\infty$ law for the Hausdorff and packing measures of the level sets of the local dimension.  Compared to the above mentioned works, apart the source of new difficulties coming from the relaxation of the regularity properties of the potentials, our assumptions provide a more general process of construction of the random Cantor set in terms of the distribution of the random family of intervals used to refine the construction at a given step: it can contain contiguous intervals (i.e. without gap in between, and even no gap) with positive probability; thus, for instance, it covers the natural families of Cantor sets one can obtain by picking at random a fiber in a Bedford-McMullen carpet. Extensions of our results to the higher dimensional case will be discussed in Remark~\ref{higherdim}. We focus on the one dimensional case because our model will be used in a companion paper to study the multifractal nature of discrete measures obtained as ``inverse'' of the random weak Gibbs measures considered here. 

Section 2 developes background about random dynamical systems and thermodynamic formalism, and presents our main results, namely  theorems~\ref{multifractal initial} and ~\ref{multifractal initial2}, as well as concrete examples of random attractors. Section 3  provides the basic properties that will used in the proof of this theorem in Section 4.

\section{Setting and main result}\label{section:Setting and main results}
We first need to expose basic facts from random dynamical systems and thermodynamic formalism.

\subsection{Random subshift, relativized entropy, topological pressure  and random weak Gibbs measures}\label{subsection: Random subshift, relativized entropy, topological pressure  and weak Gibbs measures}
{\bf Random subshift.} Let $(\Omega,\mathcal{F},\mathbb{P})$ be a complete probability space and $\sigma$  a $\mathbb{P}$-preserving  ergodic map. Let $l$ be a $\mathbb{Z}^+$ valued random variable (r.v.) such that
$
\int\log (l)\, \mathrm{d}\mathbb{P}<\infty$ and $ \mathbb{P}(\{\omega\in\Omega, l(\omega)\geq 2\})>0.
$
Let $A=\{A(\omega)=(A_{r,s}(\omega)):\omega\in\Omega\}$ be  a random transition matrix such that $A(\omega)$ is a $l(\omega)\times l(\sigma\omega)$-matrix with entries 0 or 1. We  suppose that the map $\omega\mapsto A_{r,s} (\omega)$ is measurable for all $(r,s)\in \mathbb{Z}^+\times \mathbb{Z}^+$ and each $A(\omega)$  has at least one non-zero entry in each row and each column.
Let $\Sigma_{\omega}=\{\vl=v_0v_1\cdots;1\leq v_k\leq l(\sigma^{k}(\omega))\text{ and } A_{v_k,v_{k+1}}(\sigma^{k}(\omega))=1\ \text{for all } k\in\mathbb{N} \},$
and $F_{\omega}:\Sigma_{\omega}\rightarrow\Sigma_{\sigma\omega}$ be the left shift $(F_{\omega}\vl)_i=v_{i+1}$ for any $\vl=v_0v_1\cdots\in \Sigma_{\omega}$. Define $\Sigma_{\Omega}=\{(\omega,\vl):\omega\in\Omega, \vl\in\Sigma_\omega\}$ and the map $\Pi:\Sigma_{\Omega}\rightarrow\Omega$ as $\Pi(\omega,\vl)=\omega$.
Define the map $F:\Sigma_{\Omega}\to \Sigma_{\Omega}$  as $F((\omega,\vl))=(\sigma\omega,F_{\omega}\vl)$. The corresponding family $\tilde{F}=\{F_{\omega}:\omega\in \Omega\}$ is called a random subshift. We assume that this  random subshift  is topologically mixing, i.e.  there exists a $\mathbb{Z}^+$-valued r.v. $M$ on $(\Omega,\mathcal{F},\mathbb{P})$ such that for $\mathbb{P}$-almost every (a.e.) $\omega$,
$
A(\omega)A(\sigma\omega)\cdots A(\sigma^{M-1}\omega)$ is positive. 

\medskip

For each $n\ge 1$, define $\Sigma_{\omega,n}$ as the set of words $v=v_0v_1\cdots v_{n-1}$ of length $n$, also denoted $|v|$, such that 
$1\leq v_k\leq l(\sigma^{k}(\omega))$ for all  $0\leq k\leq n-1$
and $A_{v_k,v_{k+1}}(\sigma^{k}(\omega))=1$ for all $ 0\leq k\leq n-2$. For $v=v_0v_1\cdots v_{n-1}\in\Sigma_{\omega,n}$,  we define the cylinder $[v]_{\omega}$ as
$[v]_{\omega}:=\{\wl\in\Sigma_\omega:w_i=v_i\text{ for }i=0,\dots,n-1\}.$
For any $s\in \Sigma_{\omega,1}$, $n\geq M(\omega)$ and $s'\in \Sigma_{\sigma^n\omega,1}$, there is  at least one word $v(s,s')\in \Sigma_{\sigma\omega,n-2}$ such that $sv(s,s')s'\in \Sigma_{\omega,n}$. For each such connection, we fix one such $v(s,s')$ and denote the word $sv(s,s')s'$ by $s\ast s'$.

\medskip

\noindent{\bf Relativized entropy.} Using the same notations as in \cite{Kifer1,Kifer4,KL}, let
$$\mathcal{P}_{\mathbb{P}}(\Sigma_\Omega)=\{\rho, \text{ probability measure on }\Sigma_\Omega:\Pi_*\rho=\rho\circ\Pi^{-1}=\mathbb{P}\},$$
and
$$\mathcal{I}_{\mathbb{P}}(\Sigma_\Omega)=\{\rho\in\mathcal{P}_{\mathbb{P}}(\Sigma_\Omega):\rho \text{ is } F\text{-invariant}\}.$$

Any $\rho\in \mathcal{P}_{\mathbb{P}}(\Sigma_\Omega) $ on $\Sigma_{\Omega}$ disintegrates in $\mathrm{d}\rho(\omega,\vl)=\mathrm{d}\rho_{\omega}(\vl)\mathrm{d}\mathbb{P}(\omega)$ where the measures $\rho_{\omega},\ \omega\in\Omega$, are regular conditional probabilities with respect to the $\sigma$-algebra $\pi_\Omega^{-1}(\mathcal F)$, where $\pi_\Omega$ is the canonical projection from $\Sigma_\Omega$ to $\Omega$.  This implies that  for $\mathbb{P}$-almost every $\omega$, for any measurable set $R\subset \Sigma_{\Omega}$, $\rho_{\omega}(R(\omega))=\rho(R|\pi_\Omega^{-1}(\mathcal F))$, where
$R(\omega)=\{x:(\omega,\vl)\in R\}$.

Let $\mathcal{R}=\{R_i\}$ be a finite or countable partition of $\Sigma_{\Omega}$ into measurable sets. Then for all $\omega\in\Omega$,  $\mathcal{R}(\omega)=\{R_i(\omega):\ R_i(\omega)=\{x\in \Sigma_{\omega}:\ (\omega,x)\in R_i \}\}$ is a partition of~$\Sigma_{\omega}$. 

Given $\rho\in \mathcal{P}_{\mathbb{P}}(\Sigma_\Omega)$, the
conditional entropy of $\mathcal{R}$  given $\pi_\Omega^{-1}(\mathcal F)$ is defined by
\begin{eqnarray*}
	H_{\rho}(\mathcal{R}|\pi_\Omega^{-1}(\mathcal F))&=&-\int \sum_{i} \rho(R_i|\pi_\Omega^{-1}(\mathcal F))\log (\rho(R_i|\pi_\Omega^{-1}(\mathcal F)))d \mathbb{P}\\
	&=&\int H_{\rho_{\omega}}(\mathcal{R}(\omega))d \mathbb{P}(\omega)
\end{eqnarray*}
where $H_{\rho_{\omega}}(\mathcal{A})$ denotes the usual entropy of a partition $\mathcal{A}$.

Now, given a finite or countable partition $\mathcal{Q}$ of $\Sigma_\Omega$, define the fiber entropy of $F$, also called  the relative entropy of $F$ with respect to  $\mathcal{Q}$, as
$$h_{\rho}(F,\mathcal{Q})=\lim_{n\to\infty}\frac{1}{n} H_{\rho}(\vee_{i=0}^{n-1}F^{-i}\mathcal{Q}|\pi_\Omega^{-1}(\mathcal F))$$
(here $\vee$ denotes the join of partitions).

Then define
$$h_\rho(F)=\sup_{\mathcal{Q}} h_{\rho}(F,\mathcal{Q}),$$
where the supremum is taken over all finite or
countable measurable partitions $\mathcal{Q}=\{Q_i\}$ of $\Sigma_{\Omega}$ with finite conditional entropy, that is 
$h_{\rho}(F,\mathcal{Q})<+\infty$. In our setting, we have $h_\rho(F)\le \int \log(l)\,\mathrm{d}\mathbb{P}$.  The number $h_\rho(F)$, also denoted $h(\rho|\mathbb{P})$ in the literature, is the relativized entropy of $F$ given $\rho$. It is   also called the fiber entropy of the bundle random dynamics~$F$.

\medskip

\noindent{\bf Topological pressure and random weak Gibbs measures.} We say that a measurable function $\Phi$ on $\Sigma_{\Omega}$ is in $\mathbb{L}^1_{\Sigma_{\Omega}}(\Omega,C(\Sigma))$ if
\begin{enumerate}
	\item \begin{equation}\label{int}
	C_{\Phi}=:\int_{\Omega} \|\Phi(\omega)\|_{\infty}\, \mathrm{d}\mathbb{P}(\omega)<\infty,
	\end{equation}
	where
	$
	\|\Phi(\omega)\|_{\infty}=:\sup_{\underline{v}\in \Sigma_{\omega}}|\Phi(\omega,\vl)|$,
	\item %\begin{equation}\label{var}
	For $\mathbb{P}$-a. e. $\omega$,  $\text{var}_n\Phi(\omega)\to 0$ as $n\to\infty$,
	where 
	$$
	\text{var}_n\Phi(\omega) = \sup\{|\Phi(\omega,\vl)-\Phi(\omega,\wl)| : v_i = w_i, \forall i < n\}.
	$$
\end{enumerate}

Now, if $\Phi\in \mathbb{L}^1_{\Sigma_{\Omega}}(\Omega,C(\Sigma))$, due to Kingsman's subadditive ergodic theorem, 
$$\lim_{n\to\infty}\frac{1}{n}\log \sum_{v\in\Sigma_{\omega,n}}\sup_{\underline v\in[v]_\omega}\exp\left(S_n\Phi(\omega,\underline v)\right )$$ 
exists for $\mathbb{P}$-a.e.  $\omega$ and does not depend on $\omega$, where $S_n\Phi(\omega,\underline v)=\sum_{i=0}^{n-1}\Phi(F^i(\omega,\underline v))$.  It is called {\it topological pressure of $\Phi$} and is denoted $P(\Phi)$. 

Also, with $\Phi$ is associated 
the Ruelle-Perron-Frobenius  operator $\lan_{\Phi}^\omega: C^0(\Sigma_{\omega})\to C^0(\Sigma_{\sigma\omega})$ defined as
$$\lan_{\Phi}^\omega h(\vl)=\sum_{F_{\omega}\wl=\vl}\exp(\Phi(\omega,\wl))h(\wl),\ \ \forall\ \vl\in \Sigma_{\sigma\omega}.$$ 
\begin{proposition}\label{eigen}\cite{Kifer1,MSU}
	Removing from $\Omega$ a set of $\mathbb{P}$-probability 0 if necessary, for all  $\omega\in \Omega$ there exists $\lambda(\omega)>0$ and  a probability measure $\widetilde \mu_{\omega}$ on $\Sigma_{\omega}$ such that $(\mathcal{L}_{\Phi}^\omega)^*\widetilde \mu_{\sigma\omega}=\lambda(\omega)\widetilde \mu_{\omega}$.
\end{proposition}
We call the  family $\{\widetilde \mu_{\omega}:\omega\in\Omega\}$ a random weak Gibbs measure on $\{\Sigma_\omega:\omega\in\Omega\}$ associated with $\Phi$.

\subsection{A model of random dynamical attractor}

For any $\omega\in\Omega$, let $U_{\omega}^1=[a_{\omega,1},b_{\omega,1}],U_{\omega}^2=[a_{\omega,2},b_{\omega,2}],\cdots U_{\omega}^{s}=[a_{\omega,s},b_{\omega,s}]\cdots$ be closed non trivial intervals with disjoint interiors and $b_{\omega,s}\le a_{\omega,s+1}$. We assume that for each $s\ge 1$, $\omega\mapsto (a_{\omega,s},b_{\omega,s})$ is measurable, as well as  $a_{\omega,1}\geq 0$ and $b_{\omega,l(\omega)}\leq 1$.  Let $f^s_\omega(x)=\frac{x-a_{\omega,s}}{b_{\omega,s}-a_{\omega,s}}$ and consider a measurable  mapping $\omega\mapsto \mathbf{T}_{\omega}^s$ from $(\Omega,\mathcal F)$ to the space of $C^1$ diffeomorphisms of $[0,1]$ endowed with its Borel $\sigma$-field. We consider  the measurable $C^1$ diffeomorphism  $T_{\omega}^s: U_{\omega}^s \rightarrow [0,1]$ by $T_{\omega}^s=\mathbf{T}_{\omega}^s\circ f^s_\omega$. We denote the  inverse of $T_\omega^s$ by $g_{\omega}^s$. We also define
\begin{eqnarray*}
	U_{\omega}^v&=&g_{\omega}^{v_0}\circ g_{\sigma\omega}^{v_1}\circ\cdots\circ g_{\sigma^{n-1}\omega}^{v_{n-1}}([0,1]),\ \forall v=v_0v_1\cdots v_{n-1}\in\Sigma_{\omega,n},\\
	X_{\omega}&=&\bigcap_{n\geq 1}\bigcup_{v\in\Sigma_{\omega,n}}U_{\omega}^v,\\
	X_{\Omega}&=&\{(\omega,x):\omega\in\Omega, x\in X_{\omega}\},
\end{eqnarray*}
and  for all $\omega\in\Omega$,  $ s\ge 1$ and $x\in U_{\omega}^{s}$, 
%\begin{equation*}
$$
\psi(\omega,s,x)=-\log|(T_{\omega}^{s})'(x)|.
$$
%\end{equation*}

We say that a measurable function $\widetilde \psi$ defined on $\widetilde{U}_{\Omega}=\{(\omega,s,x):\omega\in \Omega,1\leq s\leq l(\omega),x\in U_{\omega}^s\}$ is in $\mathbb{L}^1_{X_{\Omega}}(\Omega,\widetilde{C}([0,1]))$ if
\begin{enumerate}
	\item %\begin{equation*}
	$\int_{\Omega}\|\widetilde\psi(\omega)\|_{\infty}\mathrm{d}\mathbb{P}(\omega)<\infty$,
	%	\end{equation*}
	where
	%\begin{equation}\label{cb'}
	$\|\widetilde\psi(\omega)\|_{\infty}:=\sup_{1\leq s\leq l(\omega)}\sup_{x\in U_{\omega}^{s}}|\widetilde\psi(\omega,s,x)|$,
	%\end{equation}
	
	\item for $\mathbb{P}$-a.e.   $\omega\in\Omega$, $\text{var}(\widetilde\psi,\omega,\varepsilon) \to 0$ as $\varepsilon\to 0$, where 
	%\begin{equation}\label{var(...)}
	$$\text{var}(\widetilde\psi,\omega,\varepsilon)=\sup_{1\leq s\leq l(\omega)}\sup_{x,y\in U_{\omega}^s\text{ and } |x-y|\leq \varepsilon }|\widetilde\psi(\omega,s,x)-\widetilde\psi(\omega,s,y)|.$$
	%\end{equation}  
\end{enumerate}

We will make the following assumption:

\medskip $\psi\in\mathbb{L}^1_{X_{\Omega}}(\Omega,\widetilde{C}([0,1]))$ and $\psi$ satisfies the  contraction property in the mean  
\begin{equation}\label{cpsi}
c_{\psi}:=-\int_{\Omega}\sup_{1\leq s\leq l(\omega)}\sup_{x\in U_{\omega}^{s}}\psi(\omega,s,x)\mathrm{d}\mathbb{P}(\omega)>0.
\end{equation}

Under this assumption, there is $\mathbb{P}$-almost surely a natural projection $\pi_\omega: \Sigma_{\omega}\rightarrow X_{\omega}$ defined as
$$% \nonumber to remove numbering (before each equation)
\pi_\omega(\underline{v})= \lim_{n\rightarrow\infty}g_{\omega}^{v_0}\circ g_{\sigma\omega}^{v_1}\circ\cdots\circ g_{\sigma^{n-1}\omega}^{v_{n-1}}(0).$$
This mapping may not be injective, but any $x\in X_{\omega}$ has at most two preimages in $\Sigma_{\omega}$. We can define 
$
\Psi(\omega,\vl)=\psi(\omega,v_0,\pi(\vl))\text{ for }\vl=v_0v_1\cdots\in\Sigma_{\omega}$, then $ \Psi\in \mathbb{L}^1_{\Sigma_{\Omega}}(\Omega,C(\Sigma))$.
By a standard way, we can easily get that for $\mathbb{P}$-almost every $\omega\in \Omega$, the Bowen-Ruelle formula holds, i.e. $\dim_H X_{\omega}=t_0$ where $t_0$ is the unique root of the equation $P(t\Psi)=0$.

%\begin{theorem}\label{Bowen's equality}
%	Under the Assumption \ref{ass1}, for $\mathbb{P}$-almost every $\omega\in \Omega$, the Bowen-Ruelle formula holds, i.e. $\dim_H X_{\omega}=t_0$ where $t_0$ is the unique root of the equation $P(t\Psi)=0$.
%\end{theorem}

\subsection{Multifractal analysis of the random weak Gibbs measures}
Our results statement require some additional definitions related to multifractal formalism. 

Let $\mu$ be a compactly supported positive and finite Borel measure on $\R^n$.
\begin{definition}
	The (lower) $L^q$-spectrum $\tau_{\mu}:\mathbb{R}\rightarrow\mathbb{R}\cup\{-\infty\}$ and the upper-$L^q$ spectrum $\overline{\tau}_{\mu}:\mathbb{R}\rightarrow\mathbb{R}\cup\{-\infty\}$ are respectively defined by
	\begin{align}
	\label{Lql}
	\tau_{\mu}(q)&=\liminf_{r\rightarrow 0}\frac{\log \sup\{\sum_{i}(\mu(B_i))^q\}}{\log (r)}\\
	%\end{equation}\\
	%\begin{equation}
	\label{Lqu}
	\text{and }\overline{\tau}_{\mu}(q)&=\limsup_{r\rightarrow 0}\frac{\log \sup\{\sum_{i}(\mu(B_i))^q\}}{\log (r)},
	\end{align}
	where the supremum is taken over all families of disjoint closed balls $B_i$ of radius $r$ with centers in $\text{supp}(\mu)$.
\end{definition}

By construction, the function $\tau_\mu$ is non decreasing and concave over its domain, which equals $\R$ of $\R^+$ (see \cite{LN,Barral2015}). 

\begin{definition}
	The lower and upper large deviations spectra $\underline{LD}$ and $\overline{LD}$ are given by
	\begin{equation}\label{LD}
	\underline{LD}_{\mu}(d)=\lim_{\varepsilon\to 0}\liminf_{r\rightarrow 0}\frac{\log \#\{i:r^{d+\varepsilon}\leq\mu(B(x_i,r)\leq r^{d-\varepsilon})\}}{-\log (r)}
	\end{equation}
	\begin{equation}\label{LDu}
	\overline{LD}_{\mu}(d)=\lim_{\varepsilon\to 0}\limsup_{r\rightarrow 0}\frac{\log \#\{i:r^{d+\varepsilon}\leq\mu(B(x_i,r)\leq r^{d-\varepsilon})\}}{-\log (r)},
	\end{equation}
	where the supremum is taken over all families of disjoint closed balls $B_i=B(x_i,r)$ of radius $r$ with centers $x_i$ in $\supp(\mu)$.
\end{definition}

\begin{definition}
	For all $x\in \supp(\mu)$, define
	$$\dl(\mu,x)=\liminf_{r\to 0^+}\frac{\log \mu(B(x,r))}{\log r} \text{ and } \du(\mu,x)=\limsup_{r\to 0^+}\frac{\log \mu(B(x,r))}{\log r}.$$
	Then, for $d\le d'\in \R$, define
	\begin{align*}
	\El(\mu,d)&=\{x\in \supp(\mu) :\dl(\mu,x)=d\},\\
	\overline{E}(\mu,d)&=\{x\in \supp(\mu) :\du(\mu,x)=d\},\\
	E(\mu,d)&=\El(\mu,d)\cap\overline{E}(\mu,d),\\
	E(\mu,d,d')&=\{x\in \supp(\mu):\dl(\mu,x)=d,\du(\mu,x)=d'\}.
	\end{align*}
	It is clear that since $\mu$ is bounded, $E(\mu,d,d')=\emptyset$ if  $d'<0$.
	
	Finally, define
	$$\dim_H(\mu)=\sup\{s:\text{ for }\mu\text{-almost every }x\in \supp(\mu), \dl(\mu,x)\geq s\}$$
	and 
	$$\dim_p(\mu)=\sup\{s:\text{ for }\mu\text{-almost every }x\in \supp(\mu), \du(\mu,x)\geq s\}$$
	Equivalent definitions are (see~\cite{Falconer1}):
	$$\dim_H(\mu)=\inf\{\dim_H E: \,E\text{ Borel set},\, \mu(E)>0\}
	$$and 
	$$\dim_p(\mu)=\inf\{\dim_p E: \,E\text{ Borel set},\,\mu(E)>0\}.$$
	
\end{definition}

\begin{definition}(Legendre Transform)
	For any function $f:\mathbb{R}\to \mathbb{R}\cup\{-\infty\}$ with non-empty domain, its Legendre transform $f^*$ is defined on $\mathbb{R}$  by
	$$f^*(d)=\inf_{q\in \mathbb{R}}\{d q-f(q)\}\in \mathbb{R}\cup\{-\infty\}.$$
\end{definition}

One always has (see \cite{Olsen1,LN})
\begin{align*}
\dim_H E(\mu,d)&\le \min (\dim_H \El(\mu,d), \dim_H \overline E(\mu,d),\dim_P E(\mu,d))\\
&\le  \max (\dim_H \El(\mu,d), \dim_H \overline E(\mu,d),\dim_P E(\mu,d))\le \tau_\mu^*(d).
\end{align*} 
and (see \cite{Barral2015})
\begin{align}
\dim_H E(\mu,d)&\le \underline{LD}_{\mu}(d)\le \overline{\tau}^{*}_{\mu}(d) \label{LDl}\\
\dim_P E(\mu,d)&\le \overline{LD}_{\mu}(d)\le \tau^{*}_{\mu}(d)\label{LDO}.
\end{align}

\begin{definition}(Multifractal formalism)
	We say that $\mu$ obeys the  multifractal formalism at $d\in\mathbb{R}\cup\{\infty\}$ if $\dim_H E(\mu,d)=\tau_\mu^*(d)$, and that the multifractal formalism holds (globally) for $\mu$ if it holds at any $d\in\mathbb{R}\cup\{\infty\}$ (here a negative dimension means that the set is empty).
	
\end{definition}
The reader should have in mind that if the domain of $\tau_\mu$ is the whole interval~$\R$, then $\tau_\mu^*(d)\ge 0$ if and only if $\tau_\mu^*(d)>-\infty$, i.e. $d\in [\tau_\mu'(+\infty),\tau_\mu'(-\infty)]$.  Also, if the multifractal formalism holds at $d$, then 
$$\dim_H E(\mu,d)=\dim_{P}E(\mu,d)=\underline{LD}_{\mu}(d)=\overline{LD}_{\mu}(d)=\tau_{\mu}^\ast(d).
$$

\medskip

Let $\phi\in \mathbb{L}^1_{\Sigma_{\Omega}}(\Omega,\widetilde{C}([0,1]))$
and consider the function 
$$
\Phi(\omega,\vl)=\phi(\omega,v_0,\pi(\vl)) \quad (\vl=v_0v_1\cdots\in\Sigma_{\omega}).
$$
We have $\Phi\in\mathbb{L}^1_{\Sigma_{\Omega}}(\Omega,C(\Sigma))$. Let $\mu$ be the random weak Gibbs measure on $\{X_\omega:\omega\in \Omega\}$ obtained as $\mu_{\omega}={\pi_{\omega}}_* \widetilde \mu_\omega:= \widetilde \mu_\omega \circ \pi_\omega^{-1}$, where $\widetilde\mu$ is obtained from proposition~\ref{eigen}. Without changing the random measures $\widetilde\mu_\omega$ and $\mu_\omega$,  we can  assume $P(\Phi)=0$.  Then, since equation~\eqref{cpsi}, for any $q\in \mathbb{R}$, there exists a unique ${T}(q)\in\R$ such that 
$
P(q\Phi-{T}(q)\Psi)=0,
$ 
and the mapping $T$ is concave and non decreasing. 
\medskip

\begin{theorem}\label{multifractal initial}
	For $\mathbb{P}$-a.e.  $\omega\in\Omega$,
	\begin{enumerate}
		\item for all $q\in \R$, $\tau_{\mu_{\omega}}(q)=\overline\tau_{\mu_\omega}(q)=T(q)=\displaystyle\min_{\rho\in\mathcal{I}_{\mathbb{P}}(\Sigma_\Omega)}\left\{\frac{h_\rho(F)+q \int \Phi \mathrm{d}\rho}{\int \Psi \mathrm{d}\rho}\right\}.$
		
		\item The  multifractal formalism holds globally for $\mu_\omega$.	Furthermore, for all $d\in [T'(+\infty),T'(-\infty)]$, one has
		$$\dim_H E(\mu_\omega,d)=T^*(d)=\max_{\rho\in\mathcal{I}_{\mathbb{P}}(\Sigma_\Omega)}\left \{-\frac{h_\rho(F)}{\int \Psi \mathrm{d}\rho}: \dfrac{\int \Phi \mathrm{d}\rho}{\int \Psi \mathrm{d}\rho}=d \right\}.$$
		\item For all $d\le  d'\in [T'(+\infty),T'(-\infty)]$,
		\begin{align*}\dim_H E(\mu_{\omega},d,d')&=\inf\{T^*(d),T^{*}(d')\}\\
		\text{and }\dim_{P} E(\mu_{\omega},d,d')&=\sup\{T^*(\beta):\beta\in [d,d']\}.
		\end{align*}
		\item  For all $d\in [T'(+\infty),T'(-\infty)]$,
		\begin{align*}\dim_H \El(\mu_{\omega},d)&=T^*(d),\ \dim_{P} \El(\mu_{\omega},d)=\sup\{T^*(d'):d'\geq d\},\\
		\dim_H \overline{E}(\mu_{\omega},d)&=T^*(d),\ \text{and }\dim_{P} \overline{E}(\mu_{\omega},d)=\sup\{T^*(d'):d'\leq d\}.
		\end{align*}
		
	\end{enumerate} 
\end{theorem}

We refer to \cite{Mattila} for the notions of generalized Hausdorff and packing measures $\mathcal H^g$ and $\mathcal P^g$ associated with a gauge function $g$, i.e. a function $g:[0,\infty)\to [0,\infty) $ which is  non decreasing and satisfies $g(0)=0$. 
\begin{theorem} \label{multifractal initial2}
	For all $d\in [T'(+\infty),T'(-\infty)]$ such that $T^{*}(d)<\max (T^*)$, for all gauge function $g$, the following zero-infinity laws hold: 
	\begin{align*}
	\mathcal H^g(E(\mu,d))&=\begin{cases} 0&\text{if }\limsup_{r\to 0}\frac{\log g(r)}{\log r}> T^*(d),\\
	+\infty&\text{ otherwise}
	\end{cases} \\
	\text{and }
	\mathcal P^g(E(\mu,d))&=\begin{cases} 0&\text{if }\liminf_{r\to 0}\frac{\log g(r)}{\log r}> T^*(d),\\
	+\infty&\text{ otherwise}
	\end{cases}.
	\end{align*}
\end{theorem}
Let us put our result in perspective with respect to the existing literature. 

The study achieved in  \cite{PW1,PW2}  leads to  the multifractal nature of Gibbs measures projected on some random Cantor sets whose construction assumes a strong separation condition for the pieces of the construction.  About the same time, the multifractal analysis of random Gibbs measures and Birkhoff averages  on random Cantor sets and the whole torus were obtained in \cite{Kifer2,Kifer3}; 	when the support of the measure is a Cantor set, a strong separation condition is assumed as well. More recently, in \cite{FS}, the  multifractal analysis for disintegrations of Gibbs measures on $\{1,\ldots,m\}^\N\times\{1,\ldots,m\}^\N$ was achieved as a consequence of the multifractal analysis  of conditional Birkhoff averages of random continuous potentials (not $C^\alpha$). The approach developed there could, with some effort, be adapted to derive our results on weak Gibbs measures  if we worked with random fullshift only. However, as we already said it in the beginning of the introduction, the method cannot be extended easily to the random subshift, and our view point will be different. In \cite{FS}, the authors start by establishing large deviations results, and then use them to construct by concatenation Moran sets of arbitrary large dimension in the level sets $E(\mu_\omega,d)$; we will concatenate  information provided by random Gibbs measures associated with H\"older potentials which  approximate the continuous potentials associated with the random weak Gibbs measure and the random maps generating the attractor $X_\omega$. This will provide us with a very flexible tool from which, for instance, we will  deduce the result about the sets $E(\mu_\omega,d,d')$. In this sense, our results also complete a part of those obtained in \cite{MSU} which, in particular, achieves the multifractal analysis of random Gibbs measures on  random Cantor sets obtained as the repeller of random conformal maps.

The multifractal analysis of Birkhoff averages on  random conformal repellers of $C^1$ expanding maps is studied in \cite{Shu}, where the random dynamics is in fact coded by a non random subshift of finite type, and  the random potentials that are considered  satisfy an equicontinuity property stronger than the one we require. 

The sets $E(\mu,d,d')$ were studied for deterministic  Gibbs measures on conformal repellers and for  self-similar measures in \cite{FW2001,Olsen2003,BOS,Olsen2008}.

Finally, in \cite{MWW2002}, zero-infinity laws  are established for Besicovitch subsets of self-similar sets of the line. This inspired theorem~\ref{multifractal initial2}, of which the results in \cite{MWW2002}   turn out to be a special case. Also,  in \cite{MW},  a zero-infinity law is established  for the Hausdorff and packing measure of sets of generic points of invariant measures on a conformal repeller. 

\medskip

\begin{remark}~\label{higherdim} 
	The approach used in this paper can be extended to the higher dimension if the random attractor can be represented as follows:
	For any $\omega\in\Omega$, let $U_{\omega}^1,U_{\omega}^2,\cdots, U_{\omega}^{l(\omega)}$ be closed sets which are the closures of their interiors supposed to be pairwise disjoint.  We assume that 
	\begin{enumerate}
		\item $U_{\omega}^s\subset [0,1]^d$ for any $1\le s\le l(\omega)$.
		\item There exists a measurable $C^1$ diffeomorphism  $T_{\omega}^s: U_{\omega}^s \rightarrow [0,1]^d$. We denote the  inverse of $T_\omega^s$ by $g_{\omega}^s$. 
		Furthermore, $T:\widetilde{U}_{\Omega}=\{(\omega,s,x):\omega\in \Omega,1\leq s\leq l(\omega),x\in U_{\omega}^s\}\to [0,1]^d$ which is defined as $(\omega,s,x)\mapsto T_{\omega}^s(x)$ is measurable.
		\item There exists a function $\psi\in \mathbb{L}^1_{X_{\Omega}}(\Omega,\widetilde{C}([0,1]^d))$ such that 
		$$\exp(-\text{var}(\psi,\omega,d(x,y)))\leq\frac{d(g_{\omega}^s(x),g_{\omega}^s(y))}{\exp(\psi(\omega,s,x)) d(x,y)}\leq\exp(\text{var}(\psi,\omega,d(x,y)))$$
		for any $x,y\in U_{\omega}^s$ with $1\le s\le l(\omega)$.   	  	 
		
		Here, we say that a measurable function $\widetilde \psi$ defined from $\widetilde{U}_{\Omega}$ to $\mathbb{R}$ is in $\mathbb{L}^1_{X_{\Omega}}(\Omega,\widetilde{C}([0,1]^d))$ if
		\begin{itemize}
			\item %\begin{equation*}
			$\int_{\Omega}\|\widetilde\psi(\omega)\|_{\infty}\mathrm{d}\mathbb{P}(\omega)<\infty$,
			%	\end{equation*}
			where
			%\begin{equation}\label{cb'}
			$\displaystyle \|\widetilde\psi(\omega)\|_{\infty}:=\sup_{1\leq s\leq l(\omega)}\sup_{x\in U_{\omega}^{s}}|\widetilde\psi(\omega,s,x)|$,
			%\end{equation}
			
			\item for $\mathbb{P}$-a.e.   $\omega\in\Omega$, $\text{var}(\widetilde\psi,\omega,\varepsilon) \to 0$ as $\varepsilon\to 0$, where 
			%\begin{equation}\label{var(...)}
			$$\text{var}(\widetilde\psi,\omega,\varepsilon)=\sup_{1\leq s\leq l(\omega)}\sup_{x,y\in U_{\omega}^s\text{ and } d(x,y)\leq \varepsilon }|\widetilde\psi(\omega,s,x)-\widetilde\psi(\omega,s,y)|.$$
			%\end{equation}  
		\end{itemize}
		\item Equation~\eqref{cpsi} holds.
	\end{enumerate} 
	
	Now, we can define
	\begin{eqnarray*}
		U_{\omega}^v&=&g_{\omega}^{v_0}\circ g_{\sigma\omega}^{v_1}\circ\cdots\circ g_{\sigma^{n-1}\omega}^{v_{n-1}}([0,1]^d),\ \forall v=v_0v_1\cdots v_{n-1}\in\Sigma_{\omega,n}\\
		X_{\omega}&=&\bigcap_{n\geq 1}\bigcup_{v\in\Sigma_{\omega,n}}U_{\omega}^v\\
		X_{\Omega}&=&\{(\omega,x):\omega\in\Omega, x\in X_{\omega}\}.
	\end{eqnarray*}
	
	Nevertheless there is a difference in the estimation of  the local dimensions of measures. As will see,  in this paper, a building block in our proofs is the comparison of the mass assigned by a random Gibbs measure to neighboring basic intervals of the form $U^v_\omega$, in order to control the mass assigned to centered intervals by  a random weak Gibbs measure, as well as some auxiliary measures obtained by concatenation of pieces of random Gibbs measures (this point of view is fruitful in the study of the discrete inverses of random weak Gibbs measures in the companion paper mentioned at the beginning of this introduction). In higher dimension the situation is different in general. If a strong separation condition is satisfied by the basic sets $U^v_\omega$, there is no much difference with the 1 dimensional case. Otherwise, one can adapt the method used in  \cite{Patzschke} for self-conformal measures and under the open set condition, which, for any Gibbs measure $\nu$, consists in controlling the asymptotic behavior of the distance of $\nu$-almost every point $x$ to the boundary of the basic set of the $n$-th generation  containing $x$.  We omit the details. 
\end{remark}

\subsection{Examples of random attractor} 
We end this section with  examples illustrating  our assumptions on the random attractors considered in this paper. 
As a first example, one has the fibers of McMullen-Bedford self-affine carpets, and more generally the Gatzouras-Lalley self-affine carpets \cite{LG}, which naturally illustrate the idea that at a given step of the construction two consecutive intervals $U_\omega^s$ and $U_\omega^{s+1}$ may touch each other.  In \cite{Luzia}, Luzia considers  a class of expanding maps of the 2-torus of the form $f(x,y)=(a(x,y),b(y))$ that are $C^2$-perturbations of Gatzouras-Lalley carpets, whose fibers illustrate our purpose with nonlinear maps. These examples are associated with random fullshift. Let us give a first more explicit example associated with a random subshift and a piecewise linear random maps. 

Let $\Omega=\Gamma:=\widetilde{\mathbb{Z}}^+\times\widetilde{\mathbb{Z}}^+\times\cdots,$
$\mathcal{F}$ be the $\sigma$-algebra generated by the cylinders $[n_1n_2\cdots n_k]$ ($k\in \N$,  $n_i\in\tilde{\mathbb{Z}}^+$ for  $1\leq i\leq k$), and $\mathbb{P}$ the probability measure on $(\Omega,\mathcal F)$ defined by 
$$\mathbb{P}([n_1n_2\cdots n_k])=\frac{1}{n_1(n_1+1)}\cdot\frac{1}{n_2(n_2+1)}\cdot\dots\cdot \frac{1}{n_k(n_k+1)}.$$
Also, let $\sigma$ be the shift map on $\Omega$. Such a system is ergodic. It satisfies the conditions we need.

For $\underline{n}=n_1n_2\cdots n_k\cdots\in \Gamma$,
define
$l(\underline{n})=n_1$ and 
$A(\underline{n})$ the  $n_1\times n_2$-matrix with all entries equal to 1 if $n_2\neq n_2-1$ or $n_1=2$, and the $n_1\times n_2$-matrix whose $n_1-1$ first rows have entries equal to 1 and the entries of the $n_1$-th row equal 0 except that $A_{n_1,n_1-1}(\underline n)=1$. It it is easy to check that both $l$ and $A$ are measurable, that  $\int \log l \, d \mathbb{P}<+\infty$, and $l$ and $A$ define a random subshift, which is not a fullshift. Also, the integer $M=\inf\{m\in\N: \,
A(\omega)A(\sigma\omega)\cdots A(\sigma^{M-1}\omega)\text{ is positive}\}$ is measurable, since for  any $k\in \mathbb{N}$, $$\{\underline{n}\in\Omega : M(\underline{n})=k\}=[(k+1)k(k-1)\cdots 2].$$ 
Notice that both $l$ and $M$ are unbounded. 

Then we set $T_{\underline{n}}^i(x)=n_1 x\mod 1$ for $x\in [\frac{i-1}{n_1},\frac{i}{n_1}]$ and for $i=1,2,\cdots,n_1$. 

%For any point $\omega=\underline n$, if $n_k=3$, then at the $k$-th step the whole length of the cylinders will become $5/6$ of the whole length of $k-1$-th step. As $\mathbb{P}([3])=1/12$, from Poincar\'{e}'s recurrence theorem \cite[theorem 1.4]{Walters} or ergodic theorem \cite[theorem 1.14]{Walters}, for $\mathbb{P}$-almost every  $\omega\in\Omega$, $3$ will appear infinity many times, so at last the whole length (Lebesgue measure) will be $0$ for the fiber at $\omega$.    

In fact, the measure $\mathbb P$ defined above is a special example of a Gibbs measure on  $(\Omega,\mathcal F,\sigma)$ (see \cite{Sarig1,Sarig2}). So we can enrich the previous construction by considering any such measure $\mathbb P$ for which  $\int \log l\ {\rm d} \mathbb P<+\infty$. For the mappings maps $T^s_\omega$, here is a way to provide a non trivial example, which seems to be not covered by the existing literature. 

Start with a family $\{\varphi_{s,\omega}\}_{s\in\N}$ of  random $C^1$ differeomorphisms of $[0,1]$ such that at least one $\varphi'_{s,\omega}$ is nowhere $C^\epsilon$ with positive probability. Assume that there exists a random variable $a_0$ taking values in $(0,1]$ and such that
$$
\inf_{1\leq s\leq l(\omega), \,x\in[0,1]} |\varphi'_{s,\omega}(x)|\ge a_0(\omega).
$$
%$$
%\min_{\omega\in\Omega,\, s\in\N, \,x\in[0,1]} |\varphi'_{s,\omega}(x)|\ge a_0
%$$ 
Let $T_{\omega}^s=\varphi_{s,\omega}\circ f^s_\omega$, where $f_{\omega}^s$ is the linear map from $U_{\omega}^s$ onto $[0,1]$. Then, the constant $c_\psi$ of equation~\eqref{cpsi} satisfies 
$$
c_\psi\ge \int_\Omega\left [ \log(a_0(\omega))-\sup_{1\le s\le l(\omega)} \log (|U^s_\omega|)\right]\,\mathrm{d}\mathbb P(\omega).
$$

%$$
%c_\psi\ge \log(a_0)-\int_\Omega \sup_{1\le s\le l(\omega)} \log (|U^s_\omega|)\,\mathrm{d}\mathbb P(\omega).
%$$
Thus, we require that
$$
\int_\Omega \left [\log(a_0(\omega))-\sup_{1\le s\le l(\omega)} \log (|U^s_\omega|)\right]\,\mathrm{d}\mathbb P(\omega)>0.
$$

%$$
%\log(a_0)-\int_\Omega \sup_{1\le s\le l(\omega)} \log (|U^s_\omega|)\,\mathrm{d}\mathbb P(\omega)>0.
%$$
This allows some $T^s_\omega$ be not uniformly expanding, but ensures expansiveness in the mean.    It is easily seen that the Lebesgue measure of $X_\omega$ is almost surely bounded by $\prod_{i=0}^{n-1} \left (\sum_{1\le s\le l(\omega)} |U^s_{\sigma^i\omega}|/a_0(\sigma^i\omega)\right)$ for all $n\ge 1$. Thus, if we strengthen our requirement by assuming  that 
$$
\int_\Omega \left [\log(a_0(\omega))-\log\left (\sum_{1\le s\le l(\omega)} |U^s_\omega|\right)\right] \,\mathrm{d}\mathbb P(\omega)>0,
$$ 
then the Lebesgue measure of $X_\omega$ is 0 almost surely.

Now let us provide a completely  explicit illustration of the last idea (we will work with a random fullshift for simplicity of the exposition). 

We take $(\Omega,\mathcal{F},\mathbb{P},\sigma)$ as the  fullshift  $(\{0,1,2\}^{\mathbb{N}},\mathcal{F},\mathbb{P},\sigma)$. For any $n$-th cylinder $[\omega_0\omega_1\cdots \omega_{n-1}]\subset \Omega$ we set $\mathbb{P}([\omega_0\omega_1\cdots \omega_{n-1}])=\frac{1}{3^n}$. It is the unique ergodic measure of maximal entropy for the shift map.

Let $l$ be a random variable depending on $\omega_0$ only, which is given by
$$l(\omega)=\left\{\begin{array}{ll}
4 & \omega_0=0 \\
1 & \omega_0=1\\
3 & \omega_0=2
\end{array}
\right.$$
The entries of the random transition matrix are always 1 (we consider the random fullshift). We assume that the map $T(\omega,x)$  just depends on $\omega_0$ and $x$.

If $\omega_0=0$, let $\varphi_{s,\omega}(x)=x$ for $s=1,2,3,4$ and $U_{\omega}^1=[0,1/4]$, $U_{\omega}^2=[1/4,1/2]$, $U_{\omega}^3=[1/2,3/4]$ and $U_{\omega}^4=[3/4,1]$. In this case, we know that $a_0(\omega)=1$; notice that the intervals $U_{\omega}^s,\ 1\leq s\leq 4$ cover the interval $[0,1]$. 

If $\omega_0=1$, let  $h(x)=6+\sum_{j=1}^{+\infty}j^{-2}\sin(2^j\pi x)$.
Define
$$\varphi_{1,\omega}(x)=\frac{\int_{0}^x h(t) dt}{\int_{0}^1 h(t) dt},$$
and $U_{\omega}^{1}=[0,1]$.
In this case we can choose $a_0(\omega)=1/2$. It is easy to check that $T_{\omega}^1$ is not expanding on some interval;  furthermore it is just of class $C^1$ since $h$ is nowhere $\epsilon$-H\"{o}lder for any $\epsilon\in (0,1)$.

%For any $N\in\mathbb{N}$, let $h_N(x)=\sum_{j=1}^{N}j^{-2}sin(2^{(2^j)}\pi x)$, then $\max_{x \in [0,1]} |h_N'(x)| \le \sum_{j=1}^N j^{-2}2^{(2^j)} \pi \le 2\pi*2^{(2^j)}$. For any $s=k/(2^{(2^N)})$ and $t=(k+1/2)/(2^{(2^N)})$ with $k\in \mathbb{Z}$, one can get that $h(s)=h_{N-1}(s)+6$ and $h(t)=h_{N-1}(t)+N^{-2}+6$. Since $|h_{N-1}(s)-h_{N-1}(t)|\leq 2\pi*2^{(2^{(N-1)})}*\frac{1/2}{2^{(2^N)}}=\frac{\pi}{2^{2^{(N-1)}}}$, 
%$$|h(s)-h(t)|\geq N^{-2}-\frac{\pi}{2^{2^{(N-1)}}}.$$
%Noticed that $|s-t|=\frac{1}{2*(2^{(2^N)})}$, So that $h$ could not be $\epsilon$-H\"{o}lder for any $\epsilon\in (0,1)$. 

If $\omega_0=2$, let $\varphi_{s,\omega}(x)=x$ for $s=1,3$ and  $\varphi_{2,\omega}(x)=\frac{7x}{8}+\frac{x^2}{8}$, and  $U_{\omega}^1=[0,1/9]$, $U_{\omega}^2=[1/9,2/9]$, $U_{\omega}^3=[2/3,7/9]$. 
It is easily checked that the left derivative of $T_{\omega}^1$ and the right derivative of $T_{\omega}^2$ do not coincide, so the dynamics is not the restriction of a random conformal map.  
In this case we can choose $a_0(\omega)=7/8$.

Also,
$$\int_\Omega\left[ \log(a_0(\omega))-\log\left (\sum_{1\le s\le l(\omega)} |U^s_\omega|\right)\right] \,\mathrm{d}\mathbb P(\omega)=\frac{\log 21-\log 16}{3}>0,$$
so that all the conditions hold.

\section{Basic properties of random weak Gibbs and random Gibbs measures. Approximation of $(\Phi,\Psi)$ by random  H\"older potentials}\label{BPORWGM'}

This section prepares the proofs of our main results. Sections 3.1 and 3.2 present basics facts about random weak Gibbs and random Gibbs measures. Section 3.3.1 provides an approximation of $(\Phi,\Psi)$ by a family $\{(\Phi_i,\Psi_i)\}_{i\ge 1}$ of random  H\"older potentials. Section 3.3.2 derives some related properties of the associated pressure functions, which yield the variational formulas appearing in theorem~\ref{multifractal initial}. Section 3.3 presents properties related to  the random Gibbs measures associated with the couple $(\Phi_i,\Psi_i)$, which  will be used as building blocks in the concatenation of measures used in the proof of the main parts of theorem~\ref{multifractal initial} and of theorem~\ref{multifractal initial2} (section 4).

\subsection{Properties of weak Gibbs measures} Fix a potential $\Phi\in \mathbb{L}^1_{\Sigma_{\Omega}}(\Omega,C(\Sigma))$ (here $P(\Phi)$ may not be 0). Since $\var_n\Phi(\omega)\to 0$ as $n\to 0$ and $(\var_n\Phi)_{n\ge1}$ is bounded in $L^1$ norm,  using Maker's  ergodic theorem~\cite{Maker}, we  can get
\begin{equation}\label{Vn}
\lim_{n\to\infty}\frac{1}{n}V_n\Phi(\omega)=0,\quad \mathbb{P}\text{-almost surely,}
\end{equation}
where $V_n\Phi(\omega):=\sum_{i=0}^{n-1}\var_{n-i}\Phi(\sigma^i\omega)=o(n)$.
Due to  (\ref{int}) and the ergodic theorem, setting $S_n\|\Phi(\omega)\|_{\infty}=\sum_{i=0}^{n-1}\|\Phi(\sigma^i\omega)\|_{\infty}$, for any positive sequence $(a_n)_{n\ge 0}$ such that $a_n=o(n)$, $\mathbb{P}$-almost surely we have
\begin{equation}\label{Phi}
\big|S_n\|\Phi(\omega)\|_{\infty}-S_{n-a_n}\|\Phi(\omega)\|_{\infty}\big |=nC_{\Phi}-(n-a_n)C_{\Phi}+o(n)=o(n).
\end{equation}

\begin{definition}
	A family $u=\{u_{n,\omega}:\Sigma_{\omega,n}\rightarrow\Sigma_{\omega}\}_{n\in\N}$ of measurable maps satisfying $(u_{n,\omega}(v))|_n=v$ for all $(n,\omega)\in \mathbb{N}\times\Omega$ and $v\in\Sigma_{\omega,n}$  is called an extension. We say that it is measurable if the mapping $(\omega,x)\mapsto u_{n,\omega}(x)$ is measurable for all $n\in\mathbb{N}$.
\end{definition}
Let $u=\{u_{n,\omega}\}$ be an extension and $\Phi\in \mathbb{L}^1_{\Sigma_{\Omega}}(\Omega,C(\Sigma))$. Then for $(n,\omega)\in \mathbb{N}\times\Omega$
\begin{equation*}
Z_{u,n}(\Phi,\omega):=\sum_{v\in\Sigma_{\omega,n}}\exp\big (S_n\Phi(\omega,u_{n,\omega}(v))\big)
\end{equation*}
is called $n$-th partition function of $\Phi$ in $\omega$ with respect to $u$.

Due to the assumption $\log (l)\in\mathbb{L}^1(\Omega,\mathbb{P})$, using the same method as in \cite{Gundlach,KL}, it is easy to prove the following lemma.

\begin{lemma}\label{converge}
	Let $u$ be any extension and $\Phi\in \mathbb{L}^1_{\Sigma_{\Omega}}(\Omega,C(\Sigma))$.
	
	Then $\lim_{n\rightarrow\infty}\frac{1}{n}\log Z_{u,n}(\Phi,\omega)=P(\Phi)$ for $\mathbb{P}$-almost every $\omega\in\Omega$. This limit is independent of $u$.
	
	Furthermore, the variation Principle holds, which means,
	\begin{equation}\label{variation}
	P(\Phi)=\sup_{\rho\in\mathcal{I}_{\mathbb{P}}(\Sigma_\Omega)}\left \{h_\rho(F)+\int \Phi \mathrm{d}\rho\right \}.
	\end{equation} 
\end{lemma}

Now let 
$$
\lambda(\omega,n)=\lambda(\omega)\cdot\lambda(\sigma\omega)\cdot\dots\cdot\lambda(\sigma^{n-1}\omega),
$$ 
where $\lambda(\omega)$ is defined as in proposition~\ref{eigen}. The following lemma is direct when the potential $\Phi$ possesses bounded distorsions so that the Ruelle-Perron-Frobenious theorem holds for the operator $\mathcal{L}_{\Phi}^{\omega}$. For general potentials in $\mathbb{L}^1_{\Sigma_{\Omega}}(\Omega,C(\Sigma))$ we need a proof.

\begin{lemma}\label{lambda}
	One has $\displaystyle \lim_{n\to\infty}{\frac{\log \lambda(\omega,n)}{n}}=P(\Phi)$  for $\mathbb{P}$-almost every  $\omega\in\Omega$.
\end{lemma}

\begin{remark} In the next proof, as well as in the rest of the paper, we will use the letter $M$ to denote the levels of the function $M(\cdot)$. Keeping this in mind should prevent from some confusion.
\end{remark}

\begin{proof}
	For $M>0$, let $F_M=\{\omega\in \Omega: M(\omega)\leq M\}$. Fix  $M$ large enough so that  $\mathbb{P}(F_{M})>0$. For each $\omega\in\Omega$, let $b_k(\omega)$ be the $k$-th return time of $\omega$ to the set $F_{M}$, i.e. $b_1=\inf\{n\in \N:\sigma^n\omega\in F_{M}\}$ and for $k>1$, $b_k=\inf\{n\in \N:\sigma^n\omega\in F_{M},n>b_{k-1}\}$. From ergodic theorem we get $\lim_{k\to\infty}\frac{b_k}{k}=\frac{1}{\mathbb{P}(F_M)}$ for $\mathbb{P}$-almost every  $\omega\in \Omega$. Then $\lim_{k\to\infty}\frac{b_{k+1}-b_k}{b_k}=0$ for $\mathbb{P}$-almost every  $\omega\in \Omega$, which implies that  $M(\sigma^n\omega)= o(n)$.
	
	Now we claim that for any $\vl \in \Sigma_{\sigma^n\omega}$ we have 
	$$
	Z_{u,n-M(\sigma^n\omega))}(\Phi,\omega)\exp(-o(n))\leq \mathcal{L}_{\Phi}^{\omega,n} 1(\vl)\leq Z_{u,n}(\Phi,\omega)\exp(o(n)).$$ The second  inequality uses the fact that we work with a subshift as well as \eqref{Vn}. We just prove the first  inequality: for $n$ large enough so that $M(\sigma^n\omega)\le n$, 
	\begin{eqnarray*}
		\mathcal{L}_{\Phi}^{\omega,n} 1(\vl) &=& \sum_{w\in \Sigma_{\omega,n},w\vl\in\Sigma_{\omega}}\exp(S_n\Phi(\omega,w\vl)) \\
		&\ge & \sum_{w'\in \Sigma_{\omega,n-M(\sigma^n\omega)}}\exp\big (S_{n-M(\sigma^n\omega)}(\omega,u_{\omega,n-M(\sigma^n\omega)}(w'))-o(n)\big )\\
		&=&Z_{u,n-M(\sigma^n\omega))}(\Phi,\omega)\exp(-o(n)).
	\end{eqnarray*}
	By using the topological mixing property and  preserving for each $w'\in \Sigma_{\omega,n-M(\sigma^n\omega)}$ only one path of length $M(\sigma^n\omega)$ from $w'$ to $\underline v$, the inequality follows from   (\ref{Vn}), $M(\sigma^n\omega)=o(n)$ and (\ref{Phi}).
	
	Now, since $\lambda(\omega,n)=\int\mathcal{L}_{\Phi}^{\omega,n} 1(\vl) d \widetilde \mu_{\sigma^n\omega}(\vl)$,  we can easily get the result from lemma \ref{converge} and the fact that  $M(\sigma^n\omega)=o(n)$.
	
\end{proof}

For each $\omega \in\Omega$, let $$D(\omega)=:\frac{1}{\lambda(\omega,M(\omega))}\exp(-S_{M(\omega)}\|\Phi(\omega)\|_{\infty}).$$

Then $D(\omega)>0$ for $\mathbb{P}$-almost every  $\omega\in \Omega$.
From $M(\sigma^n\omega)=o(n)$, \eqref{Phi} and lemma \ref{lambda} we can get that $\log D(\sigma^n\omega)=o(n)$ $\mathbb{P}$-almost surely.

Recall that by proposition~\ref{eigen}, for $\mathbb{P}$-a.e $\omega\in \Omega$, the measures $\widetilde\mu_\omega$ satisfy $(\mathcal{L}_{\Phi}^\omega)^*\widetilde\mu_{\sigma\omega}=\lambda(\omega)\widetilde \mu_\omega$.
\begin{proposition}\label{wgibbs}
	For $\mathbb{P}$-a.e $\omega\in\Omega$, for any $n\in\mathbb{N}$, for all $v=v_0v_1\dots v_{n-1}\in \Sigma_{\omega,n}$, one has
	$$\frac{D(\sigma^n\omega)}{\lambda(\omega,n)}\exp{(\inf_{\vl\in[v]_{\omega}}S_n\Phi(\omega,v))}\leq \widetilde \mu_{\omega}([v]_{\omega})\leq \frac{1}{\lambda(\omega,n)}\exp{(\sup_{\vl\in[v]_{\omega}}S_n\Phi(\omega,\vl))},$$
	so that
	$$\exp(-\epsilon_n n)\leq\frac{\widetilde \mu_{\omega}([v]_{\omega})}{\exp(S_n\Phi(\omega,\vl)-\log(\lambda(\omega,n)))}\leq \exp(\epsilon_n n)$$
	for any $\vl\in [v]_{\omega}$, where $\epsilon_n$ does not depend on $v$ and tends to 0 as $n\to\infty$.
\end{proposition}
\begin{proof}
	Let us deal first with the case $n=1$.
	
	Fix $1\leq i\leq l(\omega)$. For any $1\leq j\leq l(\sigma^{M(\omega)}\omega)$, there exists $w\in\Sigma_{\sigma\omega,M(\omega)-1}$ such that $iwj\in\Sigma_{\omega,1+M(\omega)}$.
	Due to proposition~\ref{eigen}, we have  
	$$
	\widetilde\mu_{\omega}([iwj]_{\omega})=\frac{1}{\lambda(\omega,M(\omega))}\int \mathcal{L}_{\Phi}^{\omega,M(\omega)} 1_{[iwj]_{\omega}} \, \mathrm{d} \widetilde\mu_{\sigma^{M(\omega)}\omega},$$ 
	where $\mathcal{L}_{\Phi}^{\omega,n}=\mathcal{L}_{\Phi}^{\sigma^{n-1}\omega}\circ\dots\circ\mathcal{L}_{\Phi}^{\sigma\omega}\circ\mathcal{L}_{\Phi}^{\omega}$.
	This implies $$\widetilde\mu_{\omega}([iwj])\geq \frac{\inf_{\wl\in [iwj]_{\omega}}\exp(S_{M(\omega)}\Phi(\omega,\wl))}{\lambda(\omega,M(\omega))} \widetilde\mu_{\sigma^{M(\omega)}\omega}([j]_{\sigma^{M(\omega)}\omega}).$$ Then $\widetilde\mu_{\omega}([i])\geq D(\omega)$ follows after summing over $1\leq j\leq l(\sigma^{M(\omega)}\omega)$. The upper bound $\widetilde\mu_{\omega}([i])\le 1$ is obvious.
	
	The general case is achieved similarly: If $v\in \Sigma_{\omega,n}$, for each $1\le j\le l(\sigma^{n+M(\sigma^n\omega)-1}\omega)$,  there exists $w\in\Sigma_{\sigma^n\omega,M(\sigma^n\omega)-1}$ such that $vwj\in\Sigma_{\omega,n+M(\sigma^n\omega)}$. One has
	$$
	\widetilde\mu_{\omega}([vwj]_{\omega})=\frac{1}{\lambda(\omega,n) \lambda(\sigma^n\omega,M(\sigma^n\omega))}\int \mathcal{L}_{\Phi}^{\omega,n+M(\sigma^n\omega)} 1_{[vwj]_{\omega}} \,\mathrm{d} \widetilde\mu_{\sigma^{n+M(\sigma^n\omega)}\omega},
	$$
	from which  we get
	$$
	\widetilde\mu_{\omega}([vwj]_{\omega})\ge \frac{1}{\lambda(\omega,n)} \cdot D(\sigma^n\omega) \exp{(\inf_{\vl\in[v]_{\omega}}S_n\Phi(\omega,v))}  \widetilde\mu_{\sigma^{n+M(\sigma^n\omega)}\omega}([j]_{\sigma^{n+M(\sigma^n\omega)}\omega}).
	$$
	Then, taking the sum over  $1\le j\le l(\sigma^{n+M(\sigma^n\omega)-1}\omega)$ we get 
	$$\widetilde\mu_{\omega}([v]_{\omega})\geq\frac{D(\sigma^n\omega)}{\lambda(\omega,n)}\exp{(\inf_{\vl\in[v]_{\omega}}S_n\Phi(\omega,v))}.$$
	The inequality
	$
	\widetilde \mu_{\omega}([v]_{\omega})\leq \frac{1}{\lambda(\omega,n)}\exp{(\sup_{\vl\in[v]_{\omega}}S_n\Phi(\omega,\vl))}$ is direct from the equality $$\widetilde \mu_\omega([v]_\omega)=\frac{1}{\lambda(\omega,n)} \int  \mathcal{L}_{\Phi}^{\omega,n} 1_{[v]_{\omega}}\, \mathrm{d} \widetilde\mu_{\sigma^{n}\omega}.$$
	
	Finally we conclude with  (\ref{Vn}) and $\log D(\sigma^n\omega)=o(n)$.
\end{proof}

For any $\gamma\in \mathbb{L}^1_{X_{\Omega}}(\Omega,\widetilde{C}([0,1]))$ and any $z\in U_{\omega}^v$, let 
$$
S_n\gamma(\omega,z)=\sum_{i=0}^{n-1}\gamma(\sigma^{i}\omega,v_i,T_{\sigma^{i-1}\omega}^{v_{i-1}}\cdots T_{\omega}^{v_0}z).
$$ 

From the Lagrange's finite-increment theorem, distortions and proposition \ref{wgibbs}, using standard estimates we can get the following proposition. 

\begin{proposition}\label{for n}
	For $\mathbb{P}$-almost every  $\omega\in \Omega$, there are  positive sequences $(\epsilon(\psi,n))_{n\ge 0}$ and $(\epsilon(\phi,n))_{n\ge 0}$, that we also denote as $(\epsilon(\Psi,n))_{n\ge 0}$ and $(\epsilon(\Phi,n))_{n\ge 0}$, converging to 0 as $n\to +\infty$, such that for all $n\in \mathbb{N}$, for all $v=v_0v_1\dots v_{n}\in\Sigma_{\omega,n}$, we have :
	\begin{enumerate}
		\item  For all $z\in \widering{U}_{\omega}^v,$
		$$
		\exp(S_n\psi(\omega,z)-n\epsilon(\psi,n))\leq |U_{\omega}^v|\leq \exp(S_n\psi(\omega,z)+n\epsilon(\psi,n)),$$
		hence for all $\vl\in [v]_{\omega}$,
		$$\exp(S_n\Psi(\omega,\vl)-n\epsilon(\Psi,n))\leq |U_{\omega}^v|\leq \exp(S_n\Psi(\omega,\vl)+n\epsilon(\Psi,n)).$$ 
		Consequently, for all $\vl\in X_{\omega}^v$:
		$$
		|X_{\omega}^v|\leq|U_{\omega}^v|\leq \exp(S_n\Psi(\omega,\vl)+n\epsilon(\Psi,n)).$$

		\item For all $\vl\in[v]_{\omega}$,
		$$
		\exp(S_n\Phi(\omega,\vl)-n\epsilon(\Phi,n))\leq \widetilde\mu_{\omega}([v]_{\omega})\leq \exp(S_n\Phi(\omega,\vl)+n\epsilon(\Phi,n)),$$
		hence for all $z\in U_{\omega}^v$, 
		$$
		\exp(S_n\phi(\omega,z)-n\epsilon(\phi,n))\leq \mu_{\omega}(X_{\omega}^v)= \mu_{\omega}(U^v_\omega),$$ as well as $\mu_{\omega}(U^v_\omega)\le \exp(S_n\phi(\omega,z)+n\epsilon(\phi,n))$ if $\widetilde\mu_\omega$ is atomless.
	\end{enumerate}
\end{proposition}

\subsection{Properties of random Gibbs measures}%%\label{section:Useful facts about random Gibbs measure}
Random Gibbs measures are associated with  random H\"{o}lder continuous potentials. We say that a function $\Phi$ is a random H\"{o}lder potential if $\Phi$ is measurable from $\Sigma_{\Omega}$ to $\mathbb{R}$, 
$\int \sup_{\vl\in \Sigma_{\omega}}|\Phi(\omega,\vl)| d \mathbb{P}<\infty,$ 
and there exists $\kappa\in (0,1]$ such that
\begin{equation*}%%\label{varh}
\text{var}_n\Phi(\omega)\leq K_{\Phi}(\omega)e^{-\kappa n},
\end{equation*}
where the random variable $K_{\Phi}= K_\Phi(\omega) > 0$ is such that $\int\log K_\Phi (\omega)\mathrm{d}\mathbb{P}(\omega)<\infty$.
A random H\"{o}lder continuous potential is obviously in $\mathbb{L}^1_{\Sigma_{\Omega}}(\Omega,C(\Sigma))$.

\begin{theorem}[\cite{Kifer5,KL}]%%\label{exist}
	Assume that $\mathcal{F}$ is a countably generated $\sigma$-algebra, $F$ is a topological mixing subshift of finite type  and $\Phi$ a random H\"{o}lder potential.
	
	For $\mathbb{P}$-almost every  $\omega\in\Omega$,
	there exists some random variables $C=C^{\Phi}(\omega)>0$, $\lambda=\lambda^{\Phi}(\omega)>0$, a function $h=h(\omega)=h(\omega,\vl)>0$ and a measure $\widetilde\mu\in \mathcal{M}_\mathbb{P}^1(\Sigma_{\Omega})$ with disintegrations $\widetilde\mu_{\omega}$ satisfying
	$$\int |\log C^{\Phi}|\,\mathrm{d}\mathbb{P}<+\infty,\ \ \int |\log \lambda^{\Phi}|\,\mathrm{d}\mathbb{P}<+\infty\text{ and } \log h \text{ is a H\"{o}lder potential},$$
	and such that
	\begin{equation}\label{transfer operator}
	\mathcal{L}_{\Phi}^\omega h(\omega)=\lambda(\omega) h(\sigma\omega),\  (\mathcal{L}_{\Phi}^\omega)^*\widetilde\mu_{\sigma\omega}=\lambda(\omega)\widetilde\mu_{\omega},\int h(\omega)\,\mathrm{d}\widetilde\mu_{\omega}=1.
	\end{equation}
	
	Let $m_{\omega}=m^{\Phi}_{\omega}$ be given by $dm_{\omega}=h(\omega)d\widetilde\mu_{\omega}$ and set $dm(\omega,\vl)=d\widetilde\mu_{\omega}(\vl)d \mathbb{P}(\omega)$. Then $m\in \mathcal{I}_\mathbb{P}^1(\Sigma_{\Omega})$, and  for $\mathbb{P}$-almost every  $\omega\in \Omega$,
	for all $v=v_0v_1\dots v_{n-1}\in\Sigma_{\omega,n}$, and for all $\vl\in[v]_{\omega}$
	\begin{equation*}%%\label{gibbs}
	\frac{1}{C^{\Phi}}\leq\frac{m_{\omega}([v]_{\omega})}{\exp(\sum_{i=0}^{n-1}\Phi(F^i(\omega,\vl))-\log\lambda^{\Phi}(\sigma^{n-1}\omega)\cdot\dots
		\cdot\lambda^{\Phi}(\omega))}\leq C^{\Phi}.
	\end{equation*}
	The family of measures $(m_{\omega})_{\omega\in\Omega}$ is called a random (or relative) Gibbs measure  for the potential $\Phi$.
	Moreover, $m$ is the unique maximizing $F$-invariant probability measure in the variational principle, i.e. such that
	\begin{equation*}%%\label{thermodynamic}
	P(\Phi)=h_m(F)+\int \Phi\,\mathrm{d}m,\text{ and one has }P(\Phi)=\int \log \lambda(\omega)\,\mathrm{d}\mathbb{P}.
	\end{equation*}
	
	Each time we need to refer to the function $\Phi$, we denote the measures $m$ and $m_{\omega}$ as $m^{\Phi}$ and $m^{\Phi}_{\omega}$, and denote $\lambda$ as $\lambda^{\Phi}$.
\end{theorem}
We can also define a random Gibbs measure on the random  attractor $X_\omega$ by setting $\mu_{\omega}=m_\omega \circ \pi_\omega^{-1}.$

%In this paper if the potential $\Phi$ (which is related to $\phi$) is a random H\"{o}lder potential, then when we say the relative measures $m^{\Phi}_{\omega}$ and $\mu_{\omega}^{\phi}$, they are referred to be random Gibbs measures.  

Given a random H\"{o}lder potential $\Phi$, from \eqref{transfer operator} we can  define the normalized potential $\Phi'(\omega,\vl)=\Phi(\omega,\vl)+\log h(\omega,\vl)-\log h(F(\omega,\vl))-\log \lambda(\omega)$, which satisfies $\mathcal{L}_{\Phi'}^\omega1=1$ for $\mathbb{P}$-almost every  $\omega\in \Omega$. This implies that $\Phi'\leq 0$ for $\mathbb{P}$-almost every $\omega\in \Omega$. Also,  we have the following fact:
\begin{proposition}\label{control phi}
	Suppose that $\Phi$ is a random H\"older potential.   If $P(\Phi)=0$, there exist some $\varpi>0$ such that for $\mathbb{P}$-almost every $\omega\in \Omega$, there exists $N(\omega)$ such that for any $n\geq N(\omega)$ and any $v\in \Sigma_{\omega,n}$, one has
	$$\sup_{\underline{v}\in[v]_{\omega}}S_{n}\Phi(\omega,\underline{v})\leq -n\varpi.$$
	As a consequence, $\mu_\omega$ is atomless.
\end{proposition}
If we need to refer explicitly to $\Phi$, we will use the notations $N_{\Phi}(\omega)$ and $\varpi_{\Phi}$  instead of $N(\omega)$ and $\varpi$.

The main idea of the proof is from \cite{FS}.
\begin{proof}
	Since $P(\Phi)=0$, we have $\sup\{\int \Phi\,\mathrm{d}\rho:\rho\in\mathcal{I}_\mathbb{P}(\Sigma_{\Omega})\}\leq 0$.
	
	We claim that $\sup\{\int \Phi\,\mathrm{d}\rho:\rho\in\mathcal{I}_\mathbb{P}(\Sigma_{\Omega})\}<0$.
	Let $M$ large enough such that $\mathbb{P}(\{\omega:M(\omega)<M, l(\omega)\ge 2\})>0$. For any $\omega\in\Omega$ such that $M(\omega)<M$ and $l(\omega)\ge 2$ we have $\mathcal{L}_{\Phi'}^{\omega,M}1=1$, hence $S_M{\Phi'}(\omega,\vl)<0$ for any $\vl\in \Sigma_{\omega}$ and $\int S_M{\Phi'}(\omega,\vl) d \rho_{\omega}<0$ for any probability measure $\rho_{\omega}$ on $\Sigma_{\omega}$. Since, moreover, we have  $S_M{\Phi'}\leq 0$, we conclude  that $\sup\{\int {\Phi'}\,\mathrm{d}\rho,\rho\in\mathcal{I}_\mathbb{P}(\Sigma_{\Omega})\}<0.$
	
	Let $-2\varpi:=\sup\{\int \Phi\,\mathrm{d}\rho:\rho\in\mathcal{M}_\mathbb{P}^1(\Sigma_{\Omega},F)\}$.  If the conclusion of the proposition does not hold, there exists a subsequence $(n_k)_{k\geq 1}$ such that $\#\{v: |v|=n_k, \frac{\sup_{\vl\in[v]_{\omega}}S_{n_k}\Phi(\omega,\vl)}{n_k}> -\varpi\}\geq 1$, hence for any $q\in \mathbb{R}^+$,
	\begin{eqnarray*}
		P(q\Phi) &=& \lim_{k\to\infty}\frac{\log\sum_{v\in\Sigma_{\omega,n_k}}\exp\big (qS_{n_k}\Phi(\omega,u_{n_k,\omega}(v))\big )}{n_k} \\
		&\geq&  \lim_{k\to\infty}\frac{-q\varpi n_k}{n_k}=-q\varpi.
	\end{eqnarray*}
	However,
	\begin{eqnarray*}
		P(q\Phi) &=& \sup_{\rho\in\mathcal{M}_\mathbb{P}^1(\Sigma_{\Omega},F)}\left \{h_{\rho}(F)+\int q\Phi\,\mathrm{d}\rho\right \} \\
		&\leq& \sup_{\rho\in\mathcal{M}_\mathbb{P}^1(\Sigma_{\Omega},F)}\left\{\int q\Phi\,\mathrm{d}\rho\right \}+\sup_{\rho\in\mathcal{M}_\mathbb{P}^1(\Sigma_{\Omega},F)}\{h_{\rho}(F)\}\\
		&\leq& -2q\varpi+\sup_{\rho\in\mathcal{M}_\mathbb{P}^1(\Sigma_{\Omega},F)}\{h_{\rho}(F)\}
	\end{eqnarray*}
	Since $\sup_{\rho\in\mathcal{M}_\mathbb{P}^1(\Sigma_{\Omega},F)}\{h_{\rho}(F)\}=\int\log l \,\mathrm{d}\mathbb{P}<\infty$, letting $q$ tend to $\infty$  we get a contradiction.
\end{proof}

\subsection[Approximation of $(\Phi,\Psi)$ and related properties]{Approximation of $(\Phi,\Psi)$ by random H\"{o}lder potentials, and related properties }%%\label{subsection:Approximation and related properties}

We mainly  introduce objects and related properties  which will be used in the following sections. Also, we explain the variational formulas appearing in the statement of theorem~\ref{multifractal initial}.

\subsubsection{Approximation of $(\Phi,\Psi)$ by random H\"{o}lder potentials}%%\label{subsubsection:Approximation by Holder}
Now we approximate the potentials $\Phi$ and $\Psi$ associated with $\{\mu_\omega\}_{\omega\in\Omega}$ and $\{X_\omega\}_{\omega\in\Omega}$ by more regular potentials: for any $i\ge 1$, for any $\omega\in \Omega$ for any $\vl=v_0v_1\cdots v_i\cdots\in [v]_{\omega}\subset \Sigma_{\omega}$ with $v=v_0v_1\cdots v_{i-1}\in \Sigma_{\omega,i-1}$ define
$$\Phi_i(\omega,\vl)=\frac{\max\{\Phi(\omega,\wl),\wl\in [v]_{\omega}\}+\min\{\Phi(\omega,\wl),\wl\in [v]_{\omega}\}}{2},$$
$$\Psi_i(\omega,\vl)=\frac{\max\{\Psi(\omega,\wl),\wl\in [v]_{\omega}\}+\min\{\Psi(\omega,\wl),\wl\in [v]_{\omega}\}}{2}.$$

These functions  $\Phi_i$ and $\Psi_i$ are piecewise constant with respect to the second variable. They are random H\"{o}lder continuous potentials. If we take 
$$
K_{\Phi_i}(\omega)=(2\sup_{\vl\in \Sigma_{\omega}}|\Phi(\omega,\vl)|+1)e^i\quad\text{and}\quad \kappa=1, 
$$ 
then
$$
%\begin{equation*}
\var_{n}\Phi_i(\omega)\begin{cases}\leq 
2\sup_{\vl\in \Sigma_{\omega}}|\Phi(\omega,\vl)|\leq K_{\Phi_i}(\omega)\exp(-n)& \text{if }n\leq i\\
=0&\text{if }n> i
\end{cases}.
%\end{equation*}
$$
Furthermore, $$\log ((2\sup_{\vl\in \Sigma_{\omega}}|\Phi(\omega,\vl)|+1)e^i)\leq i+2\sup_{\vl\in \Sigma_{\omega}}|\Phi(\omega,\vl)|, $$ and the right hand side  is integrable since $\Phi\in
\mathbb{L}^1_{\Sigma_{\Omega}}(\Omega,C(\Sigma))$. Also, since for $\mathbb{P}$-almost every  $\omega$ we have  $\var_n\Phi(\omega)\to 0$ as $n\to +\infty$, and $\|\Phi(\omega)-\Phi_i(\omega)\|_\infty\le \mathrm{var}_i\Phi(\omega)$, we have  $\Phi_i\to \Phi$ uniformly as $i\to \infty$ for $\mathbb{P}$-almost every  $\omega$. The same property holds for $\Psi_i$ and $\Psi$. Consequently, without loss of generality we can also assume that $P(\Phi_i)=0$ since $P(\Phi_i)$ converges to $P(\Phi)$ as $i$ tends to $+\infty$. 

\subsubsection{Approximation of ($T,T^*$) by ($T_i,T_i^*$)}%%\label{subsubsection:Approximation of}
Due to our assumptions on $(\Phi,\Psi)$ and the definition of $(\Phi_i,\Psi_i)_{i\in\N}$, we have $c_{\Psi_i}>0$, hence for the same reason as for $(\Phi,\Psi)$, for any $q\in \mathbb{R}$, for any $i\in\N$,  there exists a unique $T_i(q)$ such that $P(q\Phi_i-T_i(q)\Psi_i)=0$ and the function $T_i$ is concave and non-decreasing. Also, the function $T_i$ is differentiable since for H\"{o}lder potentials the associated random Gibbs measure is the unique invariant measure that maximizes the variation principle (see \cite{Gundlach,Kifer5,MSU}.)
\begin{lemma}\label{converge T}
	For any $q\in \mathbb{R}$, one has that  $T_i(q)\to T(q)$ as $i\to \infty$.
\end{lemma}
\begin{proof}
	At first, we recall that for any $\Phi\in \mathbb{L}^1_{\Sigma_{\Omega}}(\Omega,C(\Sigma))$  one has 
	$$P(\Phi)=\sup_{\rho\in\mathcal{I}_{\mathbb{P}}(\Sigma_\Omega)}\{h_{\rho}(F)+\int \Phi \, \mathrm{d}\rho\}.$$
	Also, for any $q\in \mathbb{R}$, we have $P(q\Phi-T(q)\Psi)=P(q\Phi_i-T_i(q)\Psi_i)=0$. Thus
	\begin{equation}\label{firstineq}
	\inf_{\rho\in\mathcal{I}_{\mathbb{P}}}\left (\int \big [q(\Phi-\Phi_i)-T(q)(\Psi-\Psi_i)-(T(q)-T_i(q))\Psi_i \big ]\, \mathrm{d}\rho\right )\leq 0,
	\end{equation}
	and
	\begin{equation}\label{secondineq}
	\sup_{\rho\in\mathcal{I}_{\mathbb{P}}}\left (\int \big [q(\Phi-\Phi_i)-T(q)(\Psi-\Psi_i)-(T(q)-T_i(q))\Psi_i\big  ]\, \mathrm{d}\rho\right )\geq 0.
	\end{equation}
	The inequality \eqref{firstineq} implies that for any $\varepsilon>0$, there exists a measure $\rho\in\mathcal{I}_{\mathbb{P}}(\Sigma_\Omega)$ such that
	$$
	\int \big [q(\Phi-\Phi_i)-T(q)(\Psi-\Psi_i)-(T(q)-T_i(q))\Psi_i  \big ]\, \mathrm{d}\rho \leq \varepsilon.$$
	Then
	$$\int (T(q)-T_i(q))\Psi_i \, \mathrm{d}\rho\geq \int q(\Phi-\Phi_i)-T(q)(\Psi-\Psi_i)\, \mathrm{d} \rho-\varepsilon,$$
	and
	\begin{eqnarray*}
		&&(T(q)-T_i(q)) \leq  \frac{\int q(\Phi-\Phi_i)-T(q)(\Psi-\Psi_i)d \rho-\varepsilon}{\int (\Psi_i)\, \mathrm{d}\rho}\\
		&&\leq  \frac{\int -|q|(\var_{i}\Phi)-|T(q)|(\var_i\Psi)\, \mathrm{d} \mathbb{P}-\varepsilon}{\int (\Psi_i)\, \mathrm{d}\rho}
		\leq  \frac{\int -|q|(\var_{i}\Phi)-|T(q)|(\var_i\Psi)\, \mathrm{d}\mathbb{P}-\varepsilon}{-c_{\Psi}/2}
	\end{eqnarray*}
	since $\int (\Psi_i)\, \mathrm{d}\rho\leq \frac{-c_{\Psi}}{2}<0$ for $i$ large enough.	Letting $i\to \infty$, from the arbitrariness of $\varepsilon$  we get $$\liminf_{i\to \infty} T_i(q)\geq T(q).$$
	Using \eqref{secondineq} similarly we can get $\limsup_{i\to \infty}T_i(q)\leq T(q).$ Finally $\lim_{i\to \infty} T_i(q)= T(q).$
	
\end{proof}
\begin{lemma}\label{convergent T*}
	Let $\widetilde T:\R\to\R$ be a concave function. Suppose that $(\widetilde T_i)_{i\ge 1}$ is a sequence of differentiable concave functions from $\R$ to $\R$ which converges pointwise to $\widetilde T$. Then $(\widetilde T_i^*)_{i\ge 1}$ converges pointwise to $\widetilde T^*$ over the interior of the domain of $\widetilde T^*$.
\end{lemma}
\begin{proof}
	Let $\alpha$ be an interior point of $\mathrm{dom}(\widetilde T^*)$. Let $q_\alpha\in \R$ be the unique point such that  $\alpha\in[\widetilde T'(q_\alpha+),\widetilde T'(q_\alpha-)]$, and $\widetilde T^*(\alpha)=\alpha q_\alpha-\widetilde T(q_\alpha)$.
	
	By \cite[proposition 2.5(i)]{Feng1}, there exists a sequence $(q_i)_{i\ge 1}$ such that for $i$ large enough one has $\widetilde T_i'(q_i)=\alpha$. Without loss of generality we can assume that this sequence converges to $q'_0\in \R$ or diverges to $-\infty$ or $\infty$.
	
	Suppose first that it converges to $q'_0\in\R$. If $q'_0=q_\alpha$ then we are done since $(\widetilde T_i)_{i\ge 1}$ converges uniformly on compact sets. Suppose that  $q'_0\neq q_\alpha$ and $q'_0>q_\alpha$. Using the uniform convergence of $(\widetilde T_i)_{i\ge 1}$ in a compact neighborhood of $ [q_\alpha,q'_0]$ and the inequality $\widetilde T_i(q)\le\widetilde  T_i(q_i)+\widetilde T'_i(q_i)(q-q_i)$ ($\widetilde T_i$ is concave), we can get $\widetilde T(q_\alpha)\le \widetilde T(q'_0)+\alpha(q_\alpha-q'_0)$. On the other hand, $T$ being concave we have $\widetilde T(q_\alpha)+\widetilde T'(q_\alpha+)(q'_0-q_\alpha)\ge  \widetilde T(q'_0)$ and $\widetilde T'(q_\alpha+)\le \alpha$. This implies that $\alpha={\widetilde T}'(q_\alpha+)$ hence $\widetilde T^*(\alpha)=\alpha q_\alpha-\widetilde T(q_\alpha)= \alpha q'_0-\widetilde \widetilde T(q'_0)=\lim_{i\to\infty} (\alpha q_i -\widetilde T_i(q_i)=\widetilde T_i^*(\alpha))$.
	
	The case  $q'_0\neq q_\alpha$ and $q'_0<q_\alpha$ is similar. Now suppose that $(q_i)_{i\ge 1}$ diverges to $\infty$ (the case where it diverges to $-\infty$ is similar). If $\widetilde T$ is affine over $[q_\alpha,\infty)$ with slope $\alpha$,  $\alpha$ is not an interior point of $\mathrm{dom}(\widetilde T^*)$. Consequently, there exists $q'_0$ and $\epsilon>0$ such that $\widetilde T'(q'_0+)<\alpha-\epsilon$, and $\widetilde T(q)\le \widetilde T(q'_0)+(\alpha-\epsilon)(q-q'_0)$ for all $q\ge q'_0$. On the other hand, since $\widetilde T'_i$ is non increasing for all $i$, for $i$ large enough we have $\widetilde T_i(q)\ge \widetilde T_i(q'_0)+\alpha (q-q'_0)$ for all $q\in[q'_0,q_i]$. Since $(q_i)_{i\ge 1}$ diverges to $\infty$, this contradicts the convergence of $(\widetilde T_i)_{i\ge 1}$ to $\widetilde T$.
\end{proof}

\subsubsection{Explanation of some variational formulas in theorem~\ref{multifractal initial}}

\begin{lemma}\label{pro:for T}
	For any $q\in\R$ we have  
	\begin{equation}\label{T(q)}
	T(q)=\min_{\rho\in\mathcal{I}_{\mathbb{P}}(\Sigma_\Omega)}\left\{\frac{h_\rho(F)+q \int \Phi \mathrm{d}\rho}{\int \Psi \mathrm{d}\rho}\right\}.
	\end{equation}
	Furthermore, for any $d\in [T'(+\infty),T'(-\infty)]$ we have 	
	\begin{equation}\label{T^{*}(d)}
	T^{*}(d)=\max_{\rho\in\mathcal{I}_{\mathbb{P}}(\Sigma_\Omega)}\left \{-\frac{h_\rho(F)}{\int \Psi \mathrm{d}\rho}: \dfrac{\int \Phi \mathrm{d}\rho}{\int \Psi \mathrm{d}\rho}=d \right\}.
	\end{equation}
\end{lemma}
\begin{proof}
	The proof of~\eqref{T(q)} is just use the fact of the variation principle, see \eqref{variation}. 
	Regarding  the equation~\eqref{T^{*}(d)}, on the one hand, for any $d\in\R$, 
	\begin{eqnarray*}
		T^*(d)&=&\inf_{q\in \R}\{qd-T(q)\}=
		\inf_{q\in \R}\left\{qd-\inf_{\rho\in\mathcal{I}_{\mathbb{P}}(\Sigma_\Omega)}\left\{\frac{h_\rho(F)+q \int \Phi \mathrm{d}\rho}{\int \Psi \mathrm{d}\rho}\right\}\right\}\\
		&=&\inf_{q\in \R}\left\{\sup_{\rho\in\mathcal{I}_{\mathbb{P}}(\Sigma_\Omega)}\left\{\frac{-h_\rho(F)-q \int \Phi \mathrm{d}\rho}{\int \Psi \mathrm{d}\rho}+qd\right\}\right\}\\
		&=&\inf_{q\in \R}\left\{\sup_{\rho\in\mathcal{I}_{\mathbb{P}}(\Sigma_\Omega)}\left\{\frac{-h_\rho(F)}{\int \Psi \mathrm{d}\rho}+q\left(d-\frac{\int \Phi \mathrm{d}\rho}{\int \Psi \mathrm{d}\rho}\right)\right\}\right\}\\
		&\geq&\sup_{\rho\in\mathcal{I}_{\mathbb{P}}(\Sigma_\Omega)}\left\{\inf_{q\in \R}\left\{\frac{-h_\rho(F)}{\int \Psi \mathrm{d}\rho}+q\left(d-\frac{\int \Phi \mathrm{d}\rho}{\int \Psi \mathrm{d}\rho}\right)\right\}\right\}\\
		&=&\sup_{\rho\in\mathcal{I}_{\mathbb{P}}(\Sigma_\Omega)}\left\{\frac{-h_\rho(F)}{\int \Psi \mathrm{d}\rho}: \frac{\int \Phi \mathrm{d}\rho}{\int \Psi \mathrm{d}\rho}=d\right\}.
	\end{eqnarray*}
	
	On the other hand, for any $d\in (T'(+\infty),T'(-\infty))$, by the proof of lemma~\ref{convergent T*}, there exists $i$ large enough and $q_i\in \R$ such that $T_i'(q_i)=d$ and 
	$P(q_i\Phi_i-T_i(q_i)\Psi_i)=0$.
	Then there exists $\rho_i\in \mathcal{I}_{\mathbb{P}}(\Sigma_\Omega)$ such that
	$h_{\rho_i}(F)+\int (q_i\Phi_i-T_i(q_i)\Psi_i) \mathrm{d}\rho=0$
	and 
	$T_i'(q_i)=d=\frac{\int \Phi_i \mathrm{d}\rho_i}{\int \Psi_i \mathrm{d}\rho_i}.$
	
	This implies that there exists $\rho_i\in \mathcal{I}_{\mathbb{P}}(\Sigma_\Omega)$ such that
	$\frac{h_{\rho_i}(F)}{\int \Psi_i \mathrm{d}\rho_i}=T^*_i(d)$
	and 
	$d=\frac{\int \Phi_i \mathrm{d}\rho_i}{\int \Psi_i \mathrm{d}\rho_i}.$
	Lemma~\ref{convergent T*} tells us that $T^*_i(d)\to T^*(d)$ as $i\to \infty$ and $\mathcal{I}_{\mathbb{P}}(\Sigma_\Omega)$ is compact for the weak* topology. Thus,  there exists  a limit point $\rho'$ of $(\rho_i)$ in $ \mathcal{I}_{\mathbb{P}}(\Sigma_\Omega) $ such that 
	$\frac{h_{\rho'}(F)}{\int \Psi \mathrm{d}\rho'}\geq T^*(d)
	\text{ and } 
	d=\frac{\int \Phi \mathrm{d}\rho'}{\int \Psi\mathrm{d}\rho'},$
	since the entropy map is upper semi-continuous and  $(\Phi_i,\Psi_i)$ converges uniformly to~$(\Phi,\Psi)$. Finally, we get 
	$$T^{*}(d)=\max_{\rho\in\mathcal{I}_{\mathbb{P}}(\Sigma_\Omega)}\left \{-\frac{h_\rho(F)}{\int \Psi \mathrm{d}\rho}: \dfrac{\int \Phi \mathrm{d}\rho}{\int \Psi \mathrm{d}\rho}=d \right\}.$$
	
	The case $d\in\{T'(+\infty),T'(-\infty)\}$ now follows by approximating $d$ by a sequence $(d_k)_{k\ge 0}$ of elements of $(T'(+\infty),T'(-\infty))$ and for each $k$ picking $\rho_k$ which realizes $\displaystyle \max_{\rho\in\mathcal{I}_{\mathbb{P}}(\Sigma_\Omega)}\left \{-\frac{h_\rho(F)}{\int \Psi \mathrm{d}\rho}: \dfrac{\int \Phi \mathrm{d}\rho}{\int \Psi \mathrm{d}\rho}=d_k \right\}$. Then, since $T^*$ is continuous at $d$ (it is lower semi-continuous as a concave function and upper semi-continuous as a Legendre transform), any limit point of $(\rho_k)_{k\ge 0}$ is such that  $\displaystyle -\frac{h_\rho(F)}{\int \Psi \mathrm{d}\rho}=T^*(d)$ and $\dfrac{\int \Phi \mathrm{d}\rho}{\int \Psi \mathrm{d}\rho}=d$.  It exists since $\mathcal{I}_{\mathbb{P}}(\Sigma_\Omega)$ is compact in the weak* topology (see \cite{Kifer4,KL}). 
\end{proof}

\subsubsection{Simultaneous control for random Gibbs measures associated with $(\Phi_i,\Psi_i)$}\label{subsection:Simultaneous control of auxiliary  random measures}

In this quite technical subsection, we prepare the ``concatenation of random Gibbs measures'' approach that will be used in the next sections to construct auxiliary measures with nice properties. We also show an almost everywhere almost doubling property for the random Gibbs measures on the random attractor $X_\omega$.

Let $D$ be a dense and countable subset of $(T'(+\infty),T'(-\infty))$. Let $\{D_i\}_{i\in \mathbb{N}}$ be an increasing sequence of finite sets such that $\cup_{i\in \mathbb{N}}D_i=D$.

Fix a sequence  $\{\varepsilon_{i}\}_{i\in \mathbb{N}}$ decreasing to 0 as $i\to \infty$.
Due to lemma \ref{converge T} and the proof of Lemma~\ref{convergent T*},  for any $i\in \mathbb{N}$ there exists $j_i$ large enough such that for any $d\in D_i$, there exists $q_i=q_i(d)\in \mathbb{R}$ such that the following properties hold:
\begin{enumerate}
	\item $T'_{j_i}(q_i)=d$, $|T_{j_i}^*(d)-T^*(d)|\leq \varepsilon_i$.
	\item $\int_{\Omega} \var_{j_i}\Phi\ d\mathbb{P}\leq \varepsilon^3_i$ and
	$\int_{\Omega} \var_{j_i}\Psi\ d\mathbb{P}\leq \varepsilon^3_i.$
\end{enumerate}

We can also assume that $j_{i+1}>j_i$ for each $i\in\mathbb{N}$.
We set 
$$
Q_i=\{q_{i}:d_i\in D_i\}.
$$

For any $q\in Q_i$, we define 
$
\Lambda_{i,q}=q\Phi_{j_i}-T_{j_i}(q)\Psi_{j_i}.
$
For the convenience of writing, we will denote $(\Phi_{j_i}, \Psi_{j_i}, T_{j_i})$ by $(\Phi_{i}, \Psi_{i}, T_{i})$, so that  
$$\Lambda_{i,q}=q\Phi_{i}-T_{i}(q)\Psi_{i}.$$
We also define $\var'_{i}\Phi(\omega):=\var_{j_i}\Phi(\omega)$. 

\medskip

Recall proposition~\ref{for n}. For each $i\geq 1$, for any $\epsilon_i\in (0,1)$, there exist positive integers $L, C, M_i, N_i$ (large enough) and $\varpi_i>0$ such that there exist a set $\Omega(i)$ and  a sequence $\{c_{i,n}\}_{n\geq 1}$ decreasing to $0$ as $n\to \infty$ such that:
$\mathbb{P}(\Omega(i))>1-\epsilon_i/4$, and for any $\omega\in\Omega(i)$, one has:
\begin{itemize}
	
	\item $M(\omega)<M_i,\ l(\omega)\leq L$,
	\item  for any $n\ge 1$, 
	$$\max (V_{n}\Phi(\omega), V_{n}\Psi(\omega))\leq nc_{i,n}$$
	and
	$$\max\{\epsilon(\psi,n)=\epsilon(\Psi,n), \epsilon(\phi,n)=\epsilon(\Phi,n)\}\leq c_{i,n},$$
	where we have used Egorov's theorem;
	\medskip
	
	\item for all $n\geq N$, for $\Upsilon\in\{\Phi,\Psi\}$ 
	\begin{eqnarray*}
		\left|S_n \var'_{i}\Upsilon(\omega)-n\int_{\Omega} \var'_{i}\Upsilon(\omega)\ d\mathbb{P}\right| &\leq & nc_{i,n},\\
		\max\left(\frac{1}{n}S_n\|\Upsilon(\omega)\|_{\infty},\ \frac{1}{n}S_n\|\Upsilon(\sigma^{-n+1}\omega)\|_{\infty}\right) &\leq & C,
	\end{eqnarray*}
	$$
	\sup_{\vl\in [v]_{\omega}} S_n\Psi(\omega,\vl)\leq (-n\varpi_{\Psi}), \  \forall v\in \Sigma_{\omega,n},
	$$
	and
	$$\left |\frac{1}{n}S_n (\log l)(\omega)\right |\leq C,$$
	where we have applied ergodic theorem and Egorov's theorem. 
	\item  for any $q\in Q_i$, the measure $\{\tilde{\mu}_{\sigma^{M}\omega}^{\Lambda_{i,q}}\}$ is  well defined, and for any $n\geq N_i$
	$$V_{n}\Lambda_{i,q}(\sigma^{M}\omega)\leq nc_{i,n} \text{ and }
	\epsilon(\Lambda_{i,q},n)\leq c_{i,n};$$ 
	$$
	\sup_{\vl\in [v]_{\sigma^{M}\omega}} S_n \Lambda_{i,q}(\sigma^{M}\omega,\vl)\leq (-n\varpi_{i}).
	$$
\end{itemize}

Let $\theta'(i,\omega,s)$ be the $s$-th return time of the point $\omega$ to the set $\Omega(i)$ under the map $\sigma$, that is 
$$\theta'(i,\omega,1)=\inf\{n\in\mathbb{N}\cup\{0\}:\sigma^n\omega\in\Omega(i)\},$$
and for any $s\in \mathbb{N}$ and $s\geq 2$,
$$\theta'(i,\omega,s)=\inf\{n\in\mathbb{N}:n>\theta'(i,\omega,s-1),\ \sigma^n\omega\in\Omega(i)\}.$$ 

Then for any $i\in\mathbb{N}$, from the ergodic theorem, for $\mathbb{P}$-almost every  $\omega$
$$\lim_{s\to \infty}\frac{\theta'(i,\omega,s)}{s}=\frac{1}{\mathbb{P}(\Omega(i))}.$$
Consequently,
\begin{equation*}%%\label{omegat}
\lim_{k\to \infty}{\frac{\theta'(i,\omega,s)-\theta'(i,\omega,s-1)}{\theta'(i,\omega,s)}}=0.
\end{equation*}
Since $\mathbb{N}$ is countable, there exists $\widetilde\Omega'\subset\Omega$ of full probability such that  for all $\omega\in\widetilde\Omega'$, for any $i\in \mathbb{N}$, we have 
$$\lim_{s\to\infty}\frac{\theta'(i,\omega,s)}{s}=\frac{1}{\mathbb{P}(\Omega(i))},$$
hence
$$\lim_{s\to\infty}\frac{\theta'(i,\omega,s)-\theta'(i,\omega,s-1)}{\theta'(i,\omega,s-1)}=0.$$

Given $\omega\in \Omega(i)$, let
$$
n^i_1(\omega)=\inf\{\theta'(i,\omega,s):\theta'(i,\omega,s)\geq M(i)\}-M(i).
$$ 
For $k\geq 2$, define $n^i_k(\omega)=\theta'(i,\omega,s_k)-M(i)$, where $s_k$ is the smallest $s$ such that the following property holds:
\begin{equation*}\label{}
\theta'(i,\omega,s)-n^i_{k-1}(\omega)\geq \max\big ( M(i), {n^i_{k-1}(\omega)}(c_{i,n^i_{k-1}})^{\frac{1}{3}}+\sqrt{\theta'(i,\omega,s)}).
\end{equation*}

It is easily seen that $$\lim_{k\to \infty}{\frac{n^i_k(\omega)-n^i_{k-1}(\omega)}{n^i_{k-1}(\omega)}}=0.$$

Now we prove an almost everywhere almost doubling property for the Gibbs measures $\mu^{\Lambda_{i,q}}_\omega$.

For $v \in \Sigma_{\sigma^{M(i)}\omega,n}$, we denote by $U_{\sigma^{M(i)}\omega}^{v+}$ and $U_{\sigma^{M(i)}\omega}^{v-}$  the two intervals of the $n$-th generation of the construction of $X_\omega$ which are neighboring~$U_{\sigma^{M(i)}\omega}^{v}$, whenever $U_{\sigma^{M(i)}\omega}^{v}$ is neither the leftmost nor the rightmost of the whole collection, and with the convention  that $U_{\sigma^{M(i)}\omega}^{v-}$ is on the left of $U_{\sigma^{M(i)}\omega}^{v}$. 

We say that $B(i,\sigma^{M(i)}\omega,k,v)$ holds if
$v \in \Sigma_{\sigma^{M(i)}\omega,n^i_k}$, and $|v\wedge v+|\leq n^i_{k-1}$ or $|v\wedge v-|\leq n^i_{k-1}$.

Let
$$
\mathcal U(i,\sigma^{M(i)}\omega,k)=\bigcup_{v\in \Sigma_{\sigma^{M(i)}\omega,n^i_k}:\,B(i,\sigma^{M(i)}\omega,k,v)\text{ holds}} U_{\sigma^{M(i)}\omega}^v.
$$

\begin{lemma}\label{control} For all  $i\in\N$,  for all $\omega\in\Omega(i)$, for all  $q\in Q_i$, we have $$\mu^{\Lambda_{i,q}}_{\sigma^{M(i)}\omega}(\bigcap_{N=1}\bigcup_{k\geq N}\mathcal U(i,\sigma^{M(i)}\omega,k))=0.$$
\end{lemma}

\begin{proof}
	For any $v$ and $v'$ such that $|v|=n^i_{k-1},|v'|=n^i_k-n^i_{k-1}$ and $vv'\in \Sigma_{\sigma^{M(i)}\omega,n^i_k}$, by construction of  $\mu^{\Lambda_{i,q}}_{\sigma^{M(i)}\omega}$ one has 
	$$\frac{\mu^{\Lambda_{i,q}}_{\sigma^{M(i)}\omega}(U_{\sigma^{M(i)}\omega}^{vv'})}{\mu^{\Lambda_{i,q}}_{\sigma^{M(i)}\omega}(U_{\sigma^{M(i)}\omega}^{v})}\leq \exp(-(n^i_k-n^i_{k-1})\varpi_i+4n^i_kc_{i,n^i_k}).$$
	We will use this fact to estimate the measure of $\mathcal U (i,\sigma^{M(i)}\omega,k)$.
	Notice that for any $v\in \Sigma_{\sigma^{M(i)}\omega,n^i_{k-1}}$, there are at most two $v'$ such that $vv'\in \Sigma_{\sigma^{M(i)}\omega,n^i_k}$ and  $B(i,\sigma^{M(i)}\omega,k,vv')$ holds. Consequently,
	\begin{equation*}
	\mu^{\Lambda_{i,q}}_{\sigma^{M(i)}\omega}(\mathcal U(i,\sigma^{M(i)}\omega,k))\leq 2\exp(-(n^i_k-n^i_{k-1})\varpi_i+4n^i_kc_{i,n^i_k}).
	\end{equation*}
	Since $\varpi_i>0$ and $n^i_k-n^i_{k-1}>n^i_{k-1}(c_{n^i_k})^{1/3}+\sqrt{n^i_k},$
	we get
	$$\sum_{k=1}^{\infty}\mu^{\Lambda_{i,q}}_{\sigma^{M(i)}\omega}(\mathcal U(i,\sigma^{M(i)}\omega,k))<+\infty.$$
	By Borel-Cantelli's lemma we get $\mu^{\Lambda_{i,q}}_{\sigma^{M(i)}\omega}(\bigcap_{N=1}\bigcup_{k\geq N}\mathcal U(i,\sigma^{M(i)}\omega,k))=0$. \end{proof}

For any $\varepsilon>0$, $\beta\geq 0$, and $k,p\geq 1$ we now define the following sets:

$$F_{i,\beta,k}(\sigma^{M(i)}\omega,{\varepsilon})=\left\{\begin{array}{l}x\in X_{\sigma^{M(i)}\omega}:
\forall \gamma\in\{-1,1\}, \forall v\in \Sigma_{\sigma^{M(i)}\omega,n^i_k}\text{ satisfying }\\
|v\wedge x|_{n^i_k}|\geq n^i_{k-1},\text{  for any } \vl\in [v]_{\sigma^{M(i)}\omega},\\
\exp(-\gamma((\beta-\gamma\varepsilon)S_{n^i_k}\Psi_{i}(\sigma^{M(i)}\omega,\vl)+\gamma S_{n^i_k}\Phi_{i}(\sigma^{M(i)}\omega,\vl))
)\leq 1
\end{array}
\right\},
$$
$$
E_{i,\beta,p}(\sigma^{M(i)}\omega,{\varepsilon})=\bigcap_{k\ge p}F_{i,\beta,k}(\sigma^{M(i)}\omega,{\varepsilon})$$
and then
$$E_{i,\beta}(\sigma^{M(i)}\omega,{\varepsilon})=\bigcup_{p\geq 1}E_{i,\beta,p}(\sigma^{M(i)}\omega,{\varepsilon}).$$

\begin{lemma}\label{initial} For all  $i\in\N$, for any $\varepsilon>0$, for all $\omega\in\Omega(i)$, for all  $q\in Q_i$, the singularity set $E_{i,T'_{i}(q)}(\sigma^{M(i)}\omega,{\varepsilon})$ has full $\mu^{\Lambda_{i,q}}_{\sigma^{M(i)}\omega}$-measure.
\end{lemma}
\begin{proof} Fix $\varepsilon>0$. Let
	$$
	S_{i,q,k}=\mu^{\Lambda_{i,q}}_{\sigma^{M(i)}\omega}\left (X_{\sigma^{M(i)}\omega}\setminus F_{i,T'_{i}(q),k}(\sigma^{M(i)}\omega,{\varepsilon})\right ).
	$$
	We have
	\begin{eqnarray*}
		S_{i,q,k}&\le &\sum_{\gamma\in \{-1,1\}}\sum_{v\in \Sigma_{\sigma^{M(i)}\omega,n^i_k}}\ \sum_{v'\in \Sigma_{\sigma^{M(i)}\omega,n^i_k},|v\wedge v'|\geq n^i_{k-1}}\mu^{\Lambda_{i,q}}_{\sigma^{M(i)}\omega}(U_{\sigma^{M(i)}\omega}^{v})\\
		&&\cdot\exp\big (-\gamma\eta((T'_{i}(q)-\gamma\varepsilon)S_{n^i_{k}}\Psi_{i}(\sigma^{M(i)}\omega,\vl')-S_{n^i_k}\Phi_{i}(\sigma^{M(i)}\omega,\vl'))\big )\\
		&=&\sum_{\gamma\in \{-1,1\}}\sum_{v,v'\in \Sigma_{\sigma^{M(i)}\omega,n^i_k},|v\wedge v'|\geq n^i_{k-1}}\exp((q+\gamma\eta)S_{n_k^i}\Phi_{i}(\sigma^{M(i)}\omega,\vl))\\
		&&\exp((-(T_{i}(q)+\gamma\eta T'_{i}(q)-\varepsilon\eta)S_{n^i_{k}}\Psi_{i})(\sigma^{M(i)}\omega,\vl))\\
		&&\cdot \exp(-\gamma\eta((T'_{i}(q)-\gamma\varepsilon)(S_{n^i_k}\Psi_{i}(\sigma^{M(i)}\omega,\vl')-S_{n^i_k}\Psi_{i}(\sigma^{M(i)}\omega,\vl)))\\
		&&\cdot\exp(\gamma\eta (S_{n^i_k}\Phi_{i}(\sigma^{M(i)}\omega,\vl')-S_{n^i_k}\Phi_{i}(\sigma^{M(i)}\omega,\vl)))+o(n^i_k))
	\end{eqnarray*}
	
	Since $T_{i}$ is in fact not only differentiable, but analytic \cite{Gundlach,MSU}, we have 	
	$$T_{i}(q+\gamma\eta)=T_{i}(q)+T_{i}'(q)\gamma\eta+O(\eta^2).$$
	uniformly in $q\in Q_i$. Thus, there exists $b>0$ such that for $\eta$ small enough, for all $q\in Q_i$, we have 
	$$|T_{i}(q+\gamma\eta)-T_{i}(q)-T_{i}'(q)\gamma\eta|\leq b\eta^2.$$ 
	
	Consider such an $\eta$ in $(0,\frac{\varepsilon}{2b}]$. We have  
	\begin{eqnarray*}
		S_{i,q,k}&\le &\sum_{\gamma\in \{-1,1\}}\sum_{v\in \Sigma_{\sigma^{M(i)}\omega,n^i_k}}(l(\sigma^{M(i)+n^i_{k-1}}\omega)\cdots l(\sigma^{M(i)+n^i_{k}-1}\omega) )\\
		&&\cdot\exp(S_{n^i_k}((q+\gamma\eta)\Phi_{i}-T_{i}(q+\gamma\eta)\Psi_{i})(\sigma^{M(i)}\omega,\vl))\\
		&&\cdot\exp((\varepsilon\eta-b\eta^2)S_{n^i_k}\Psi_{i}(\sigma^{M(i)}\omega,\vl)+o(n^i_k))\\
		&\leq& \sum_{\gamma\in \{-1,1\}}\exp((n^i_k-n^i_{k-1})C-(\varepsilon\eta-b\eta^2)n^i_{k}\varpi_{\Psi_{i}}+o(n^i_k))\\
		&\leq& \sum_{\gamma\in \{-1,1\}}\exp(n^i_kc_{i,n^i_k}-(\varepsilon\eta-b\eta^2)n^i_{k}\varpi_{\Psi_{i}}+o(n^i_k)) \quad\text{for $k$ large enough}\\
		&\leq& 2 \exp\left(-(\frac{\varepsilon^2}{4b})n^i_k\varpi_{\Psi}+o(n^i_k)\right ).
	\end{eqnarray*}
	Consequently, $\sum_{k=1}^{+\infty}S_{i,q,k}<\infty,$
	which by the Borel-Cantelli lemma yields the desired conclusion since $\varepsilon$ is arbitrary. 
\end{proof}

Now we can collect the following facts. 
\begin{fact}\label{control initial}
	{\rm 
		Lemma \ref{control} and \ref{initial} imply that  for all  $i\in\N$, for any $\epsilon_i>0$, for all $\omega\in\Omega(i)$,  there exists an integer $\mathcal{N}_i=\mathcal{N}_i(\sigma^{M(i)}\omega)$ such that for any $q\in Q_i$, there exists $E_{i,q}=E_{i,q}(\sigma^{M(i)}\omega)\subset X_{\sigma^{M(i)}\omega}$  such that
		\begin{enumerate}
			\item $\mu_{\sigma^{M(i)}\omega}^{\Lambda_{i,q}}(E_{i,q})>1-\epsilon_i$,
			\item $M(i)\leq n^i_{\mathcal{N}_i}\varepsilon_i^3$,
			\item $c_{i,n^i_{\mathcal{N}_i}}\leq \varepsilon^3_i$,
			\item $n^i_k-n^i_{k-1}\leq n^i_{k-1}\varepsilon^3_i$ for any $k\geq \mathcal{N}_i$,
			\item\label{approximate 5} for any $\vl\in E_{i,q}$, for any $v\in\Sigma_{\sigma^{M(i)}\omega,n^i_k}$ with $k\geq \mathcal{N}_i$ such that $\vl\in [v]_{\sigma^{M(i)}\omega}$, one has $|v\wedge v+|\geq n^i_{k-1}$ and $|v\wedge v-|\geq n^i_{k-1}$.
			Furthermore, for any $w\in\{v,v+,v-\}$, there exists (in fact for all ) $\wl\in [w]_{\sigma^{M(i)}\omega}$ such that
			\begin{equation}\label{TI'}
			\left|\frac{S_{n^i_k}\Phi_{i}(\sigma^{M(i)}\omega,\wl)}{S_{n^i_k}\Psi_{i}(\sigma^{M(i)}\omega,\wl)}-T'_{i}(q)\right|\leq \varepsilon_i,
			\end{equation}
			\begin{equation}\label{TImu*}
			\left|\frac{\log \mu^{\Lambda_{i,q}}_{\sigma^{M(i)}\omega}(U_{\sigma^{M(i)}\omega}^w)}{S_{|v|}\Psi_{i}(\sigma^{M(i)}\omega,\wl)}-T_{i}^*(T'_{i}(q))\right|\leq \varepsilon_i,
			\end{equation}
			and
			\begin{equation}\label{TI*}
			\left|\frac{S_{|v|}\Lambda_{i,q}(\sigma^{M(i)}\omega,\wl))}{S_{|v|}\Psi_{i}(\sigma^{M(i)}\omega,\wl)}-T_{i}^*(T'_{i}(q))\right|\leq \varepsilon_i.
			\end{equation}
			In fact with a suitable change of $\varepsilon_i$ (take it as $2\varepsilon_i$), we can get the following additional properties from \eqref{TI'},\eqref{TImu*}, and \eqref{TI*} above:
			\begin{equation}\label{TI'2}
			\left|\frac{S_{n^i_k}\Phi_{i}(\sigma^{M(i)}\omega,\wl)}{S_{n^i_k}\Psi_{i}(\sigma^{M(i)}\omega,\wl)}-d_i\right|\leq \varepsilon_i,
			\end{equation}
			\begin{equation}\label{TImu*2}
			\left|\frac{\log \mu^{\Lambda_{i,q}}_{\sigma^{M(i)}\omega}(U_{\sigma^{M(i)}\omega}^w)}{S_{|v|}\Psi_{i}(\sigma^{M(i)}\omega,\wl)}-T^*(d_i)\right|\leq \varepsilon_i,
			\end{equation}
			and
			\begin{equation}\label{TI*2}
			\left|\frac{S_{|v|}\Lambda_{i,q}(\sigma^{M(i)}\omega,\wl))}{S_{|v|}\Psi_{i}(\sigma^{M(i)}\omega,\wl)}-T^*(d_i)\right|\leq \varepsilon_i.
			\end{equation}
		\end{enumerate}

	}
\end{fact}

\begin{fact}\label{modify}
	{\rm
		We can change $\Omega(i)$ to $\Omega_i\subset \Omega(i)$ a bit smaller such that $\mathbb{P}(\Omega_i)\geq 1-\epsilon$ and	
		there exist $\mathcal{N}_i$ and $W(i)$ such that  for any $\omega\in \Omega_i$, 
		$\mathcal{N}_i(\sigma^{M(i)}\omega)\leq \mathcal{N}_i$ and $n^{i}_{\mathcal{N}_i}(\omega)\leq W(i)$
		and the properties listed in Facts \ref{control initial} hold.
	} 
\end{fact}

We  define $\theta(i,\omega,s)$ as being  the $s$-th return time to the set $\Omega_{i}$ for the point $\omega$.
Since $\mathbb{N}$ is countable, there exists $\widetilde\Omega\subset\widetilde\Omega'$ of full probability such that  for all $\omega\in\widetilde\Omega$, for any $i\in \mathbb{N}$, we have 
$$\lim_{s\to\infty}\frac{\theta(i,\omega,s)}{s}=\frac{1}{\mathbb{P}(\Omega_i)},$$
hence
$$\lim_{s\to\infty}\frac{\theta(i,\omega,s)-\theta(i,\omega,s-1)}{\theta(i,\omega,s-1)}=0.$$

\section[Multifractal of random weak Gibbs measures]{Multifractal analysis of random weak Gibbs measures:\\Proof of theorems~\ref{multifractal initial} and~\ref{multifractal initial2}}

This section consists of three subsections. In the first one we obtain the sharp lower bound for the lower $L^q$-spectrum of $\mu_\omega$. Next, in the second subsection, we prove the validity of the mutifractal formalism (see theorem~\ref{multifractal initial}(2)). Due to  \eqref{LDl}, \eqref{LDO} and lemma \ref{pro:for T}, we just need to prove $\dim_H E(\mu_{\omega},d)= \tau_{\mu_\omega}^*(d) $. There, our approach to construct suitable auxiliary measures already prepares the material used to establish in the third subsection the refinements gathered in theorem~\ref{multifractal initial}(3)(4) and theorem~~\ref{multifractal initial2}.

\subsection{Lower bound for $\tau_{\mu_\omega}$ and upper bound for $\tau_{\mu_\omega}^*$}

Fix a countable and dense subset $D$ of $\R$. Let  $\widehat{\Omega}$ be a set  of full $\mathbb P$-probability, such that:
\begin{enumerate}
	\item for all $q\in D$ the weak Gibbs measure $\{\widetilde \mu^{(q\Phi-T(q)\Psi)}_\omega\}_{\omega\in\widehat{\Omega}}$ are defined;
	\item  for all $\omega\in\widehat{\Omega}$  the conclusions of proposition~\ref{for n} hold all the potentials $q\Phi-T(q)\Psi$, $q\in D$;
	\item for all $n$ large enough, for all $v\in \Sigma_{\omega,n}$, for all $\vl\in [v]_\omega$,
\end{enumerate} 
$$
\exp(-nC_{\Psi}-o(n))\le \exp (S_n\Psi(\omega,\vl))\le \exp(-nc_{\Psi}-o(n)),
$$
which follows from  the ergodic theorem applied to the potentials $\|\Psi(\omega)\|_\infty$, where $\Psi(\omega,\vl)=\psi(\omega,v_0,\pi(\vl))\text{ for }\vl=v_0v_1\cdots\in\Sigma_{\omega}$ and then $c_\Psi>0$ and $C_\psi$ are finite.

We will establish the lower bound $\tau_{\mu_\omega}(q)\ge T(q)$ for all $\omega\in\widehat{\Omega}$ and $q\in D$. Since $D$ is dense and both $\tau_{\mu_\omega}$ and $T$ are continuous, this will yield $\tau_{\mu_\omega}\ge T$ for all $\omega\in\widehat{\Omega}$. By using the multifractal formalism, this immediately yields the desired upper bound $T^*$ for $\tau_{\mu_\omega}^*$ and the various  spectra we consider for~$\mu_\omega$. The equality $\overline\tau_{\mu_\omega}=\tau_{\mu_\omega}$ then follows from standard considerations in large deviations theory.

Fix $\omega\in\widehat{\Omega}$. Let $r>0$ and consider $\mathcal{B}=\{B_i\}$, a packing of $X_{\omega}$ by disjoint balls $B_i$ with the center $x_i$ and radius $r$.
For each ball $B_i$, choose $n=n_i$ and $v(x_i)\in \Sigma_{\omega,n}$ such that $x_i\in U_{\omega}^{v(x_i)}$ and $|U_{\omega}^{v(x_i)}|\leq r$, but $|U_{\omega}^{v(x_i)|_{n-1}}|> r$.  By removing a set of probability 0 from $\widehat{\Omega}$ if necessary, for any $\vl\in [v(x_i)]_\omega$, we have
$$r\geq|U_{\omega}^{v(x_i)}|\geq \exp(S_{n}\Psi(\omega,\vl)-o(n))\geq \exp(-nC_{\Psi}-o(n)),$$
where we have used ergodic theorem. Thus $n\geq \frac{-\log r}{2C_{\Psi}}$ for $r$ small enough.
On the other hand,  for $r$ small enough, for any $\vl\in [v(x_i)]_\omega$ we have 
$$r\leq|U_{\omega}^{v(x_i)|_{n-1}}|\leq\exp(S_{n-1}\Psi(\omega,\vl)+o(n))\leq \exp(-(n-1)c_\Psi+o(n)),$$so
$n\leq \frac{-2\log r}{c_\Psi}$. Thus, for $r$ small enough, independently on $\mathcal B$,  if $v(x_i)\in \Sigma_{\omega,n}$ and $\vl\in [v(x_i)]_\omega$ we have 
\begin{equation}\label{rn}
\frac{-\log r}{2C_{\Psi}}\le n\le \frac{-2\log r}{c_\Psi}.
\end{equation}

\medskip

\noindent
{\bf Case $q\in D\cap (-\infty,0)$:}

For each $B_i\in\mathcal{B}$, one has $X^{v(x_i)}_{\omega}\subset B_i$, so for any $\vl\in [v(x_i)]_\omega$
\begin{eqnarray*}
	(\mu_{\omega}(B_i))^q &\leq & (\mu_{\omega}(X^{v(x_i)}_{\omega}))^q\\ 
	&\le & \exp(q S_{n}\Phi(\omega,\vl)+o(n))  \\
	&=& \exp(S_{n}(q\Phi-T(q)\Psi)(\omega,\vl)) \cdot \exp(T(q)S_n\Psi(\omega,\vl)+o(n))\\
	&\leq & \mu^{(q\Phi-T(q)\Psi)}_{\omega}(X^{v(x_i)}_{\omega}) r^{T(q)} \exp(o(-\log r)),
\end{eqnarray*}
where we have applied proposition~\ref{for n}(2) to the potential $q\Phi-T(q)\Psi$ as well as proposition~\ref{for n}(1), the fact that $ |U_{\omega}^{v(x_i)|_{n-1}}|> r \ge  |U_{\omega}^{v(x_i)}|$ and \eqref{rn}. It follows that  $\sum_{i}(\mu_{\omega}(B_i))^q\leq r^{T(q)} \exp(o(-\log r)) $, and this bound does not depend on the choice of the packing $\{B_i\}$. Letting $r\to 0$ yields  $\tau_{{\mu}_{\omega}}(q)\geq T(q)$.

\medskip

\noindent
{\bf Case $q\in D\cap  [0,+\infty)$:}           Define $$V(\omega,n,r)=\{v\in\Sigma_{\omega,n}:|U_{\omega}^v|\geq 2r, \exists s \text{ such that } vs\in \Sigma_{\omega,n+1},\ |U_{\omega}^{vs}|< 2r \},$$
$$V'(\omega,n,r)=\{v\in V(\omega,n,r), \text{there is no } k<n \text{ such that } v|_k\in V(\omega,k,r)\},$$
$$V(\omega,r)=\cup_{n\geq 1}V'(\omega,n,r).$$
Then $\{U_{\omega}^v:v\in V(\omega,r)\}$ is a partition of $[0,1]$.
Define $n(\omega,r)=\max\{|v|:v\in V(\omega,r)\}$ and $n'(\omega,r)=\min\{|v|:v\in
V(\omega,r)\}$. Then, from \eqref{rn} we now that for some positive constants $B_1$ and $B_2$, for $r$ small enough, we have $-B_1 \log (r)\le n'(\omega,r)\le n(\omega,r)\le -B_2\log( r)$ .  

For any $v\in V(\omega,r)$, $U_{\omega}^v$ meets at most $\exp(o(-\log r))$ many balls of $B_i$ and for any $B_i$, $B_i$ meets at most two intervals of $U_{\omega}^v,\ U_{\omega}^{w}$ with $v,w \in V(\omega,r)$. Consequently,
since $(\mu_{\omega}(B_i))^q\leq 2^q(({\mu}_{\omega}(U_{\omega}^v))^q+({\mu}_{\omega}(U_{\omega}^{w}))^q)$, we have
\begin{eqnarray*}
	\sum_{B_i\in\mathcal{B}}({\mu}_{\omega}(B_i))^q &\leq& \exp(o(-\log r))2^q\sum_{n'(\omega,r)\leq n\leq n(\omega,r)}\sum_{v\in \Sigma_{\omega,n}\cap V(\omega,r)}({\mu}_{\omega}(U_{\omega}^{v}))^q)
\end{eqnarray*}
Using the same argument as for $q<0$, we can know get that $$({\mu}_{\omega}(U_{\omega}^{v}))^q)\leq \mu^{(q\Phi-T(q)\Psi)}_{\omega}(U_{\omega}^{v})r^{T(q)}\exp(o(-\log r)),$$
so that 
\begin{align*}
&\sum_{B_i\in\mathcal{B}}({\mu}_{\omega}(B_i))^q\\
&\leq   r^{T(q)}\exp(o(-\log r))\sum_{n'(\omega,r)\leq n\leq n(\omega,r)}\sum_{v\in \Sigma_{\omega,n}\cap V(\omega,r)}\mu^{(q\Phi-T(q)\Psi)}_{\omega}(U_{\omega}^{v})\\
&=r^{T(q)}\exp(o(-\log r)),
\end{align*}
independently on $\{B_i\}$, where we used the fact that  $\{U_{\omega}^v:v\in V(\omega,r)\}$ is a partition of $[0,1]$. Letting $r\to 0$ yields $\tau_{{\mu}_{\omega}}(q)\geq T(q)$.
\medskip

\subsection{Lower bound for the Hausdorff spectrum}%%\label{subsection:Lower bound for the Hausdorff spectrum}
Recall facts~\ref{control initial} and facts~\ref{modify} derived at the end of section~\ref{subsection:Simultaneous control of auxiliary random measures}.
For any  $\omega\in \widetilde{\Omega}$, for any $d\in [T'(+\infty),T'(-\infty)]$, for any sequence $\{d_i\}_{i\in \mathbb{N}}$ with $d_i\in D_i$, such that $\lim_{i\to \infty} d_i=d $, and consequently $\lim_{i\to \infty} T^*(d_i)=T^*(d)$ by continuity of $T^*$, we will construct a probability measure $\eta_{\omega}$ supported on a set $K(\omega,\{d_i\}_{i\geq 1})$ such that
\begin{itemize}
	%\item $\eta_{\omega}(K(\omega,\{d_i\}_{i\geq 1}))=1$,
	\item $K(\omega,\{d_i\}_{i\geq 1})\subset E(\mu_{\omega},d)$,
	\item For any $x\in K(\omega,\{d_i\}_{i\geq 1})$, $\liminf_{r\to 0}\frac{\log(\eta_{\omega}(B(x,r)))}{\log r}\geq T^\ast(d)$.
\end{itemize}
This will imply that $\dim_H \eta_{\omega}\geq T^\ast(d)$, and then $$\dim_H(E(\mu_{\omega},d))\geq \dim_H(K(\omega,\{d_i\}_{i\geq 1}))\geq T^\ast(d).$$

Fix a sequence $\{\epsilon_i\}_{i\in\N}$ small enough such that $\Pi_{i\geq1}(1-\epsilon_i)\geq \frac{1}{2}$. For each $i\in\N$, Facts \ref{control initial} will be applied with this $\epsilon_i$.

From now on:
\begin{itemize}
	\item we only deal with points $\omega$ in the set $\widetilde\Omega$  of $\mathbb P$-probability 1 for which  the sequence $\{\theta(i,\omega,s)\}_{i,s\in \N,\omega\in\widetilde\Omega}$ is well defined;
	\item the sequence $\{\Omega_i,\mathcal{N}_i,W(i),\varepsilon_i\}_{i\in \N}$ is fixed;
	\item for each $\omega\in \Omega_i$ the sequence $\{n^i_k\}_{k\in \N}$ is well defined; 
	\item we denote the properties listed in fact~\ref{control initial} by $GP(\omega,i,q)$. More precisely, we denote the five items by $GP(\omega,i,q)(1)$ to $GP(\omega,i,q)(5)$.
\end{itemize} 

We will build a family of Moran structures indexed by the  elements  of $\prod_{i\ge 1} D_i$. 

For any $\omega\in \widetilde{\Omega}$, recall that $\theta(1,\omega,1)$ is the  smallest $n\in \mathbb{N}$ such that $\sigma^{n}\omega\in \Omega_1\subset \Omega(1)$. Define $m_1:=\theta(1,\omega,1)+M(1)$.  Facts~\ref{control initial} and  facts~\ref{modify} tell us that for any $d_1\in D_1$, there exists $q_1\in Q_1$ such that $T'_{1}(q_1)=d_1$ and a set  $E_{1,q_1}(\sigma^{m_1}\omega)\subset X_{\sigma^{m_1}\omega}$ such that $GP(\sigma^{\theta(1,\omega,1)}\omega,1,q_1)$ hold.

Choose $\mathcal{N}'_1>\mathcal{N}_1$ large enough such that
\begin{itemize}
	\item $m_1\leq \varepsilon^3_{2} n^{1}_{\mathcal{N}'_1}$. 
	\item $M(2)\leq \varepsilon_2^3 n^1_{\mathcal{N}'_1}$,
	\item $W(2)\leq \varepsilon^3_2 n^{1}_{\mathcal{N}'_1}$,	
	\item for any $s$ such that the return time $\theta(2,\omega,s)$ satisfies $\theta(2,\omega,s)\geq {m_1}+n^{1}_{\mathcal{N}'_1}$,  one also has $$\frac{\theta(2,\omega,s)-\theta(2,\omega,s-1)}{\theta(2,\omega,s-1)}\leq \varepsilon^3_2.$$
	
	Let $s_{2}$ be the smallest $s$ such that $\theta(2,\omega,s)\geq m_1+n^{1}_{\mathcal{N}'_1}$.	  
\end{itemize}

Now, let $N_1$ be the largest $k$ such that $m_1+n^{1}_{k}\leq \theta(2,\omega,s_2)$ (by construction we have $N_1\ge \mathcal{N}'_1$). Then
$$\theta(2,\omega,s_2)-m_1-n^{1}_{N_1}\leq  n^{1}_{N_1+1}-n^{1}_{N_1}\le \varepsilon^3_{1} n^{1}_{N_1}$$ by item 4. above. 

Here is a picture which illustrates the beginning of the construction.
\begin{figure}[!htb]
	\begin{tikzpicture}[scale=0.7] 
	\node (b1) at (-0.1,1) {$\omega$};
	\draw[->] (-0.1,0.7) -- (-0.1,0);
	\draw[->] (0,0) -- (1,0);
	\node (b2) at (0.5,0.5)   {$n$}; 
	\node (b3) at (1,-1)   {$\sigma^{\theta(1,\omega,1)}\omega\in \Omega_1$};
	\draw[->] (1.1,-0.7) -- (1.1,0);
	\draw[->] (1.15,0) -- (2.06,0);
	\node (b4) at (1.6,0.5)   {$M(1)$}; 
	\node at (2.7,0) [circle,draw=blue!50] {$\sigma^{m_1}\omega$};
	\draw[snake=coil,segment aspect=0] (3.3,0) -- (6,0);
	\node (b5) at (4.7,1)   {long time};
	\draw[->] (4.7,0.7) -- (4.7,0.1);	
	\draw[->] (6.1,0) -- (15,0);
	\node at (1.7,0) [circle,draw=red!50,fill= red!30] {};
	\node at (3.7,0) [circle,draw=red!50,fill= red!30] {};
	\draw[snake=brace,mirror snake,raise snake=5pt,red]   (2.7,0) -- (5,0);
	\node (b6) at (3.8,-0.7)   { $n^{1}_{\mathcal{N}'_1}$}; 
	\node at (6.7,0) [circle,draw=red!50,fill= red!30] {};
	\node at (8.5,0) [circle,draw=red!50,fill= red!30] {};
	\draw[-] (10.3,0.1) -- (10.3,-0.1);
	\draw[snake=brace,snake,raise snake=5pt,red]   (2.7,0) -- (10.3,0);
	\node (b7) at (6.5,0.7)   { $n^{1}_{N_1}$};
	\node at (10.7,0) [circle,draw=red!50,fill= red!30] {};
	\node at (12.8,0) [circle,draw=red!50,fill= red!30] {};
	\draw[->] (12.8,0.7) -- (12.8,0.2);
	\node (b8) at (12.8,1)   {$\theta(2,\omega,s_2)$};
	\node at (14.2,0) [circle,draw=red!50,fill= red!30] {};
	\end{tikzpicture}
	%\caption{Relationship.}
	%%\label{fig3}
\end{figure}

For each $q\in Q_1$ and $k\ge 1$, let 
$$\mathcal{V}(\sigma^{{m_1}}\omega,1,q,k)=\left\{v\in \Sigma_{\sigma^{{m_1}}\omega ,n^{1}_{k}}
:E_{1,q}(\sigma^{m_1}\omega) \cap  X^v_{\sigma^{m_1}\omega}\neq\emptyset\right\}.$$

Also, set
$$
\mathcal{V}(\sigma^{{m_1}}\omega,1,q)=\mathcal{V}(\sigma^{{m_1}}\omega,1,q,N_1).
$$
Fix $w_0\in \Sigma_{\omega,\theta(1,\omega,1)}.$ For any $d_1\in D_1$,  we define:
$$R_1(d_1)=\left\{w_0\ast v:
v\in \mathcal{V}(\sigma^{m_1}\omega,1,q_{1})
\right\},$$
and 
$$R_1=\bigcup_{d_1\in D_1} R_1(d_1)$$
(recall that the meaning of $s*s'$ is specified at the beginning of section~\ref{subsection: Random subshift, relativized entropy, topological pressure  and weak Gibbs measures}). 
\medskip

Suppose that  $\theta({i+1},\omega,s_{i+1}),N_i,R_i$ have been chosen. Define $$m_{i+1}:=\theta({i+1},\omega,s_{i+1})+M(i+1)$$
and $$n^{i+1}_k=n^{i+1}_k(\sigma^{\theta({i+1},\omega,s_{i+1})}\omega).$$
Like in the case $i=0$, for any $d_{i+1}\in D_{i+1}$, there exists $q_{i+1}\in Q_{i+1}$ with $T_{i+1}'(q_{i+1})=d_i$ as well as a set $E_{i+1,q_{i+1}}(\sigma^{m_{i+1}}\omega)\subset X_{\sigma^{m_{i+1}}\omega}$  such that property $GP(\sigma^{\theta(i+1,\omega,s_{i+1})}\omega,i+1,q_{i+1})$ holds.

Then choose $\mathcal{N}'_{i+1}$ large enough such that
\begin{itemize}
	\item $m_{i+1}\leq \varepsilon^3_{i+2}n^{i+1}_{\mathcal N'_{i+1}}$,
	\item $M(i+2)\leq \varepsilon^3_{i+2} n^{i+1}_{\mathcal{N}'_{i+1}},$
	\item $W(i+2)\leq \varepsilon^2_{i+2} n^{i+1}_{\mathcal{N}'_{i+1}}.$ 
	
	The above two items  ensure that we do not need to wait a long relative time to go to the next step.

	\item for any $s$ with $\theta(i+2,\omega,s)\geq m_{i+1}+n^{i+1}_{\mathcal{N}'_{i+1}}$ one has $$\frac{\theta(i+2,\omega,s)-\theta(i+2,\omega,s-1)}{\theta(i+2,\omega,s-1)}\leq \varepsilon^3_{i+2}.$$
	Let $s_{i+2}$ be the smallest $s$ such that 
	\begin{equation}\label{control22}
	\theta(i+2,\omega,s)\geq m_{i+1}+n^{i+1}_{\mathcal{N}'_{i+1}}.
	\end{equation}
	
\end{itemize}

Let $N_{i+1}$ be the largest $k\ge \mathcal{N}'_{i+1}$ such that $n^{i+1}_{k}\leq \theta(i+2,\omega,s_{i+2})$. Then we have 
$$\theta(i+2,\omega,s_{i+2})-m_{i+1}-n^{i+1}_{N_{i+1}}\leq n^{i+1}_{N_{i+1}} \varepsilon^3_{i+1} $$
due to item 4. 

\begin{remark}\label{growth speed} By construction, we can take $n^{i+1}_{\mathcal{N}'_{i+1}}$ as big as we want (In the construction $m_{i+1}\leq \varepsilon^3_{i+2}n^{i+1}_{\mathcal N'_{i+1}}$). This implies that we can get the sequence $\{m_i\}_{i\in\mathbb{N}}$ increasing as fast as we want. We can also impose that $m_{i+2}-m_{i+1}\ge n^{i+1}_{\mathcal{N}'_{i+1}}$ and  $m_{i+1}=o(n^{i+1}_{\mathcal{N}'_{i+1}})$. Thus, the speed we fix for the growth of $(n^{i}_{\mathcal{N}'_{i}})_{i\in\N}$ directly impacts the growth speed of $(m_i)_{i\in\N}$.
\end{remark}

Here again we  draw a picture to illustrate this construction.
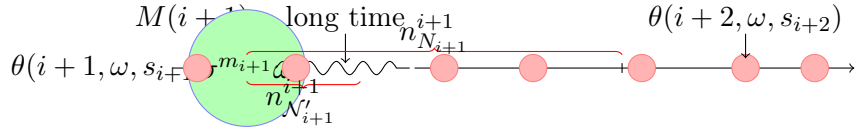
\begin{figure}[!htb]
	\begin{tikzpicture}[scale=0.65]
	\node (b3) at (-0.1,0)   {$\theta({i+1},\omega,s_{i+1})$};
	\draw[->] (1.2,0) -- (2,0);
	\node (b4) at (1.6,1)   {$M(i+1)$}; 
	\node at (2.7,0) [circle,draw=blue!50,fill= green!30] {$\sigma^{m_{i+1}}\omega$};
	\draw[snake=coil,segment aspect=0] (3.3,0) -- (6,0);
	\node (b5) at (4.7,1)   {long time};
	\draw[->] (4.7,0.7) -- (4.7,0.1);	
	\draw[->] (6.1,0) -- (15,0);
	\node at (1.7,0) [circle,draw=red!50,fill= red!30] {};
	\node at (3.7,0) [circle,draw=red!50,fill= red!30] {};
	\draw[snake=brace,mirror snake,raise snake=5pt,red]   (2.7,0) -- (5,0);
	\node (b6) at (3.8,-0.7)   { $n^{i+1}_{\mathcal{N}'_{i+1}}$}; 
	\node at (6.7,0) [circle,draw=red!50,fill= red!30] {};
	\node at (8.5,0) [circle,draw=red!50,fill= red!30] {};
	\draw[-] (10.3,0.1) -- (10.3,-0.1);
	\draw[snake=brace,snake,raise snake=5pt,red]   (2.7,0) -- (10.3,0);
	\node (b7) at (6.5,0.7)   { $n^{i+1}_{N_{i+1}}$};
	\node at (10.7,0) [circle,draw=red!50,fill= red!30] {};
	\node at (12.8,0) [circle,draw=red!50,fill= red!30] {};
	\draw[->] (12.8,0.7) -- (12.8,0.2);
	\node (b8) at (12.8,1)   {$\theta(i+2,\omega,s_{i+2})$};
	\node at (14.2,0) [circle,draw=red!50,fill= red!30] {};
	\end{tikzpicture}
	%\caption{Relationship.}
	%%\label{fig4}
\end{figure}

For $q_{i+1}\in Q_{i+1}$ and $k\ge 1$, define
$$\mathcal{V}(\sigma^{m_{i+1}}\omega,i+1,q_{i+1},k)=
\left\{v\in \Sigma_{\sigma^{m_{i+1}}\omega ,n^{i+1}_{k}}: E_{i+1,q}(\sigma^{m_{i+1}}\omega)\cap  X^v_{\sigma^{m_{i+1}}\omega}\neq\emptyset \right\},$$
and
$$\mathcal{V}(\sigma^{m_{i+1}}\omega,i+1,q_{i+1})=\mathcal{V}(\sigma^{m_{i+1}}\omega,i+1,q_{i+1},N_{i+1}).
$$

As in the case $i=0$, for any $d_{i+1}\in D_{i+1}$, we can define
\begin{eqnarray*}
	&&R_{i+1}(d_1,d_2,\cdots,d_i, d_{i+1})\\
	&&\quad\quad =\left\{w\ast  v\left|
	\begin{array}{l}
		w\in R_i(d_1,d_2\cdots d_{i}),\\
		v\in \mathcal{V}(\sigma^{m_{i+1}}\omega,i+1,q_{i+1})
	\end{array}
	\right.
	\right\}
\end{eqnarray*}
and
$$
R_{i+1}=\left\{w\ast v\left|
\begin{array}{l}
w\in R_i,\ v\in \mathcal{V}(\sigma^{m_{i+1}}\omega,i+1,q_{i+1})
\end{array}
\right.
\right\}.$$

For any $d\in [T'(+\infty),T'(-\infty)]$, there exists $\{d_i\}_{i\in \mathbb{N}}\in \prod_{i\in \mathbb{N}}D_i$, such that $\lim_{i\to\infty}d_i=d$ and
$\lim_{i\to \infty}T^*(d_i)=T^*(d)$. Moreover, if $d\in (T'(+\infty),T'(-\infty))$, $T^*_{i}({d_i})$ converges to $T^*(d)$ directly from lemma~\ref{convergent T*}. If $d\in \{T'(+\infty),T'(-\infty)\}$, again due to lemma~\ref{convergent T*}, we can choose $(d_i)_{i\ge 1}$ to be piecewise constant to make sure that $T^*_{i}({d_i})-T^*(d_i)$ tends to 0 as $i\to\infty$, so that  $T^*_{i}({d_i})$ converges to $T^*(d)$ as well. We fix such a sequence and suppose that the correspondence $q$ are $\{q_i\}_{i\in \mathbb{N}}$.

Define
$$K(\omega,\{d_i\}_{i\geq 1})=\bigcap_{i\geq 1}\bigcup_{v\in R_{i}(d_1,d_2,\cdots, d_{i})} U_{\omega}^v.$$
We will prove that for any $x\in K(\omega,\{d_i\}_{i\geq 1})$, one has
$\lim_{r\to 0+}\frac{\log \mu_{\omega}(B(x,r))}{\log r}=d.$ Then $K(\omega,\{d_i\}_{i\geq 1})\subset E(\mu_{\omega},d)$. To do so, we first establish two general estimates:

\medskip

{\bf First estimate.}
For any $w\in R_{i},\ v\in \mathcal{V}(\sigma^{m_{i+1}}\omega,i+1,q_{i+1},k) $, for any $k\ge \mathcal{N}_{i+1}$, for any $\vl\in [w\ast v]_{\omega}$   for $\Upsilon\in\{\Phi,\Psi\}$ we have  (remembering the notations introduced at the beginning of section~\ref{subsection:Simultaneous control of auxiliary  random measures} and the fact that by construction we have $| \Upsilon(\omega,\vl)- \Upsilon_{p}(\omega,\vl)|\leq \var'_{p} \Upsilon(\omega)$):
\begin{equation*}
\begin{split}
&\Big |S_{m_{i+1}+n^{{i+1}}_k} \Upsilon(\omega,\vl)-\sum_{p=1}^{i}S_{n^{p}_{N_p}} \Upsilon_{p}(F^{m_{{p}}}
(\omega,\vl))-S_{n^{{i+1}}_k} \Upsilon_{{i+1}}(F^{m_{i+1}}(\omega,\vl))\Big |  \\
\leq &\Big  |\sum_{p=1}^{i}\sum_{t=0}^{n^{p}_{N_p}-1}( \Upsilon- \Upsilon_{p})(F^{m_{p}+t}(\omega,\vl))\Big |+
\Big |\sum_{t=0}^{n^{{i+1}}_{k}-1}( \Upsilon- \Upsilon_{{i+1}})(F^{m_{i+1}+t}(\omega,\vl))\Big |\\
&\quad+\Big |\sum_{t=0}^{m_1-1} \Upsilon(F^t(\omega,\vl))\Big |+\Big | \sum_{p=1}^{i} \sum_{t=m_p+n^p_{N_p}}^{m_{p+1}-1}  \Upsilon(F^t(\omega,\vl))\Big |\\
\leq &\sum_{p=1}^{i}\sum_{t=0}^{n^p_{N_p}-1}(\var'_{p} \Upsilon)(\sigma^{m_{p}+t}\omega)+	\sum_{t=0}^{n_k^{i+1}-1}(\var'_{{i+1}} \Upsilon)(\sigma^{m_{i+1}+t}\omega)\\&+m_1C+\Big |\sum_{p=1}^{i}(m_{p+1}-m_p-n^p_{N_p})C\Big |.
\end{split}
\end{equation*}
Also, since $\int _{\Omega} \var'_{p} \Upsilon(\omega)\ d\mathbb{P}\leq \varepsilon^3_{p}$, and $\sigma^{m_{p}}\omega\in\Omega(p)$ implies $\Big|S_{n^{p}_{N_p}}\var'_{p} \Upsilon(\sigma^{m_{p}}\omega)-n^{p}_{N_p}\int _{\Omega} \var'_{p} \Upsilon(\omega)\ d\mathbb{P}\Big|\leq n^{p}_{N_p}c_{n^{p}_{N_p}}\leq n^{p}_{N_p} \varepsilon^3_{p}$, we have 
\begin{eqnarray*}
	&&\sum_{t=0}^{n^{p}_{N_p}-1}(\var'_{p} \Upsilon)(\sigma^{m_{p}+t}\omega)=S_{n^{p}_{N_p}}\var'_{p} \Upsilon(\sigma^{m_{p}}\omega)\\
	&\leq&\Big|S_{n^{p}_{N_p}}\var'_{p} \Upsilon(\sigma^{m_{p}}\omega)-n^{p}_{N_p}\int _{\Omega} \var'_{p} \Upsilon(\omega)\ d\mathbb{P}\Big|+n^{p}_{N_p}\int _{\Omega} \var'_{p} \Upsilon(\omega)\ d\mathbb{P}\\
	&\leq & 2n^{p}_{N_p}\varepsilon^3_{p}.
\end{eqnarray*}

One estimates  the term invoking $\var'_{{i+1}} \Upsilon$ similarly.  Then, recalling that $m_{p+1}-m_p-n^p_{N_p}=M(p+1)+ \theta(s+1,\omega,s_{p+1})-m_{p}-n^p_{N_p},$
and by construction 	$M(p+1)\le \varepsilon_{p+1}^3 n^{p}_{N^{p}}$ and $\theta(p+1,\omega,s_{p+1})-m_{p}-n^p_{N_p}\le 
\varepsilon_{p}^3 n^{p}_{N^{p}}$, we get 
\begin{equation}\label{6-32}
\begin{split}
&\Big |S_{m_{i+1}+n^{{i+1}}_k} \Upsilon(\omega,\vl)-\sum_{p=1}^{i}S_{n^{p}_{N_p}} \Upsilon_{p}(F^{m_{{p}}}
(\omega,\vl))-S_{n^{{i+1}}_k} \Upsilon_{{i+1}}(F^{m_{i+1}}(\omega,\vl))\Big |  \\
\leq&\sum_{p=1}^{i}n^p_{N_p}\varepsilon^3_{p}+  n^{{i+1}}_k \varepsilon^3_{i+1}+ m_1C+C\sum_{p=1}^{i}n^p_{N_p}(\varepsilon_{p}^3+\varepsilon_{p+1}^3),\\
\leq& \sum_{p=1}^{i}n^p_{N_p}((1+2C)\varepsilon^3_{p})+n^{{i+1}}_k\varepsilon^3_{i+1}+m_1C\\
\leq& \sum_{p=1}^{i-1}n^p_{N_p}((1+2C))+n^i_{N_i}((1+2C)\varepsilon^3_{i})+n^{{i+1}}_k\varepsilon^3_{i+1}+m_1C\\
\leq& m_i (1+2C)+n^i_{N_i}((1+2C)\varepsilon^3_{i})+n^{{i+1}}_k\varepsilon^3_{i+1}+m_1C\\
\leq& m_i (1+2C)+m_{i+1}((1+2C)\varepsilon^3_{i})+n^{{i+1}}_k\varepsilon^3_{i+1}+m_1C\\
\leq& (m_{i+1}+n^{{i+1}}_k)(\varepsilon^2_{i})
\end{split}
\end{equation}
for $i$ and $k$ \text{ large enough},	where in the last inequality we have used the fact that $m_i\leq m_{i+1}\varepsilon^3_{i+1}$ and $\varepsilon_{i}\geq\varepsilon_{i+1}>0$.

\medskip

{\bf Second estimate.} For $i\in\mathbb{N}$ large enough, for any $k$ with ${\mathcal{N}_{i+1}}< k \leq N_{i+1}$ for any  $v,v'\in \Sigma_{\omega,m_{i+1}+n^{i+1}_k}$ satisfying $|v\wedge v'|\geq m_{i+1}+n^{{i+1}}_{k-1}$, 
\begin{equation}\label{control neighbor}
\frac{|U_{\omega}^v|}{|U_{\omega}^{v'}|}\leq \exp((m_{i+1}+n^{{i+1}}_k) \varepsilon^2_{i}).
\end{equation}

Indeed, for $i$ large enough, we have
\begin{eqnarray*}
	&&\left|\log |U_{\omega}^v|-\log |U_{\omega}^{v'}|\right|\\
	&\leq& 2 V_{m_{i+1}+M(i+1)+n^i_{k}}\Psi(\omega)+ 2(n^{i+1}_{k}-n^{i+1}_{k-1})C\\
	&\le & 2 V_{m_{i+1}+M(i+1)+n^i_{k}}\Psi(\omega)+ 2C n_{k}^{i+1}\varepsilon_{k+1}^3\\
	&\leq& 2\sum_{i=0}^{m_1-1}\|\Psi_{\sigma^i\omega}\|_{\infty}+2\sum_{j=1}^{i}(m_{j+1}-m_j)\varepsilon_j^3+2Cn^{{i+1}}_{k}\varepsilon_{i+1}^3\\
	&\le &2\sum_{i=0}^{m_1-1}\|\Psi_{\sigma^i\omega}\|_{\infty}+2 m_i+2m_{i+1}\varepsilon_i^3+2Cn^{{i+1}}_{k}\varepsilon_{i+1}^3\\
	&\leq& 2\sum_{i=0}^{m_1-1}\|\Psi_{\sigma^i\omega}\|_{\infty}+4m_{i+1}\varepsilon_i^3+n^{{i+1}}_{k}\varepsilon_{i+1}^2;
\end{eqnarray*} 
also $2\sum_{i=0}^{m_1-1}\|\Psi_{\sigma^i\omega}\|_{\infty}\leq m_{i+1}\varepsilon_i^3$.

At last we get 
$\left|\log |U_{\omega}^v|-\log |U_{\omega}^{v'}|\right|\leq (m_{i+1}+n^{{i+1}}_k) \varepsilon^2_{i}$ for $i$ large enough. Then \eqref{control neighbor} follows.

\medskip

We can now estimate the local dimension of $\mu_\omega$. Fix $x\in K(\omega,\{d_i\}_{i\geq 1})$.  If $r$ is small enough, we can choose the largest $i$, then the largest $k=k_{i+1}$, with ${\mathcal{N}_{i+1}}< k \leq N_{i+1}$, such that the following property holds:
there exists $w\in R_{i}(d_1,d_2,\cdots, d_i),\ v\in \mathcal V(\sigma^{m_{i+1}}\omega,i+1,q_{i+1},k)$ satisfying
$x\in U_{\omega}^{w\ast v}$ and
\begin{equation*}\label{}
|U_{\omega}^{w\ast v}|\geq 2r\exp((m_{i+1}+n^{{i+1}}_k) \varepsilon^2_{i}).
\end{equation*}
From the construction, if $U_{\omega}^{w\ast v+}$ and $U_{\omega}^{w\ast v-}$ are the neighboring intervals of $U_{\omega}^{w\ast v}$, then $|v\wedge v+|$, $|v\wedge v-|$ are larger than $n^{i+1}_{k-1}$.
Then by  \eqref{control neighbor} we have $|U_{\omega}^{w\ast v+}|\geq 2r $ and $|U_{\omega}^{w\ast v-}|\geq 2r$.
So there exists $v'=v-$ or $v'=v+$ such that
$B(x,r)\subset U_{\omega}^{w\ast v}\cup U_{\omega}^{w\ast v'}.$

Now, using estimates similar to those leading to \eqref{control neighbor} with $\Psi$ replaced by $\Phi$ we get that for any $\underline{w\ast v}\in [w\ast v]_{\omega}$,
\begin{equation*}\label{}
\begin{split}
\mu_{\omega}(B(x,r)) \leq & \mu_{\omega}(U_{\omega}^{w\ast v})+\mu_{\omega}(U_{\omega}^{w\ast v'}) \\
\leq& 2\exp(S_{m_{i+1}+n^{i+1}_{k}}\Phi(\omega, \underline{w\ast v})+(m_{i+1}+n^{i+1}_{k})
\varepsilon^2_{i+1}).
\end{split}
\end{equation*}
Consequently,  using  \eqref{6-32}, 
\begin{eqnarray*}
	&&\log \mu_{\omega}(B(x,r))\\
	&\leq&\log 2+\sum_{p=1}^{i}S_{n^{p}_{N_p}}\Phi_{p}(F^{m_{{p}}}
	(\omega,\vl))+S_{n^{{i+1}}_k}\Phi_{{i+1}}(F^{m_{i+1}}(\omega,\vl))\\
	&&+2(m_{i+1}+n^{{i+1}}_k)\varepsilon^2_{i}\\
	&\leq& \sum_{p=1}^{i}S_{n^{p}_{N_p}}\Phi_{p}(F^{m_{{p}}}
	(\omega,\vl))+S_{n^{{i+1}}_k}\Phi_{{i+1}}(F^{m_{i+1}}(\omega,\vl))+3(m_{i+1}+n^{{i+1}}_k)\varepsilon^2_{i}	\end{eqnarray*}
Let $\mathcal{I}^{\Phi}_{p}=S_{n^{p}_{N_p}}\Phi_{p}(F^{m_{{p}}}(\omega,\vl))$ and $\mathcal {I}^{\Phi}_{i+1,k}=S_{n^{{i+1}}_k}\Phi_{{i+1}}(F^{m_{i+1}}(\omega,\vl))$. Then 
\begin{equation}\label{control measure B(x,r)}
\log \mu_{\omega}(B(x,r))\leq \left(\sum_{p=1}^{i}\mathcal{I}^{\Phi}_{p}\right )+ \mathcal{I}^{\Phi}_{i+1,k}+3(m_{i+1}+n^{i+1}_{k})
\varepsilon^2_{i}.
\end{equation}

Now let us estimate  $\log r$ from below:

\textbf{If $k< N_{i+1}$}, there exists $\widetilde v$ such that $|w\ast \widetilde{v}|= m_{i+1}+n^{i+1}_{k+1}$,  $x\in U_{\omega}^{w\ast\widetilde{v}}$ and 
\begin{equation}\label{6.11}
\begin{split}
|U_{\omega}^{w\ast\widetilde{v}}| \leq & 2r\exp((m_{i+1}+n^{{i+1}}_{k+1}) \varepsilon^2_{i}).
\end{split}
\end{equation}
Observe that $n^{i+1}_{k+1}-n^{i+1}_{k}\leq n^{i+1}_{k}\varepsilon^3_{i+1}$ and $\mathcal{I}^{\Psi}_{i+1,k+1}-\mathcal{I}^{\Psi}_{i+1,k}\leq -C n^{i+1}_{k}\varepsilon^3_{i+1}$ (where $I^\Psi_p$ is defined similarly as $I^\Phi_p$), so from \eqref{6.11} and  \eqref{6-32} we can get 
\begin{eqnarray*}			
	\log r&\geq& \left( \sum_{p=1}^{i}\mathcal{I}^{\Psi}_{p}\right )+ \mathcal{I}^{\Psi}_{i+1,k+1}-2(m_{i+1}+n^{i+1}_{k+1})
	\varepsilon^2_{i}-\log 2\\	
	%	&\geq&\left (\sum_{p=1}^{i}\mathcal{I}^{\Psi}_{p}\right )+ \mathcal{I}^{\Psi}_{i+1,k}--2(m_{i+1}+n^{i+1}_{k})\varepsilon^2_{i}-\log 2\\
	&\geq&\left (\sum_{p=1}^{i}\mathcal{I}^{\Psi}_{p}\right )+ \mathcal{I}^{\Psi}_{i+1,k}-C n^{i+1}_{k}\varepsilon^3_{i+1}-2(m_{i+1}+n^{i+1}_{k})\varepsilon^2_{i}-\log 2\\
	&\geq&\left (\sum_{p=1}^{i}\mathcal{I}^{\Psi}_{p}\right )+ \mathcal{I}^{\Psi}_{i+1,k}-3(m_{i+1}+n^{i+1}_{k})\varepsilon^2_{i}.
\end{eqnarray*} 
Consequently,	
\begin{equation}\label{control r}
\log r\geq \left (\sum_{p=1}^{i}\mathcal{I}^{\Psi}_{p}\right )+ \mathcal{I}^{\Psi}_{i+1,k}-3(m_{i+1}+n^{i+1}_{k})\varepsilon^2_{i}.
\end{equation}

\textbf{If $k=N_{i+1}$}, there exists $\widetilde v$ such that $|w\ast \widetilde{v}|= m_{i+2}+n^{i+2}_{\mathcal{N}_{i+2}+1}$, $x\in U_{\omega}^{w\ast\widetilde{v}}$ and 
\begin{equation}\label{}
\begin{split}
|U_{\omega}^{w\ast\widetilde{v}}| \leq& 2r\exp((m_{i+2}+n^{{i+2}}_{\mathcal{N}_{i+2}+1})\varepsilon^2_{i+1}).
\end{split}
\end{equation}
We have 
\begin{eqnarray*}
	\log r&\geq& \left (\sum_{p=1}^{i+1}\mathcal{I}^{\Psi}_{p}\right )+ \mathcal{I}^{\Psi}_{i+2,\mathcal{N}_{i+2}+1}-2(m_{i+2}+n^{i+2}_{\mathcal{N}_{i+2}+1})\varepsilon^2_{i}-\log 2\\
	&\geq& \left (\sum_{p=1}^{i}\mathcal{I}^{\Psi}_{p}\right )+ \mathcal{I}^{\Psi}_{i+1,N_{i+1}}-3(m_{i+1}+n^{i+1}_{{N}_{i+1}})\varepsilon^2_{i},
\end{eqnarray*} 
where we have used \eqref{control22}.  This implies that \eqref{control r} holds as well.

Finally, for any $\vl\in (w\ast \widetilde{v})_{\omega}$, \eqref{control measure B(x,r)} and \eqref{control r} imply 	 
\begin{equation}\label{6.13}
\frac{\log \mu_{\omega}(B(x,r))}{\log r} \geq\frac{(\sum_{p=1}^{i}\mathcal{I}^{\Phi}_{p})+ \mathcal{I}^{\Phi}_{i+1,k}+3(m_{i+1}+n^{i+1}_{k})
	\varepsilon^2_{i}}{(\sum_{p=1}^{i}\mathcal{I}^{\Psi}_{p})+\mathcal{I}^{\Psi}_{i+1,k}-3(m_{i+1}+n^{i+1}_{k})
	\varepsilon^2_{i}}
\end{equation}
Due to the construction and see \eqref{TI'2}, we have
$|\frac{\mathcal{I}^{\Phi}_{p}}{\mathcal{I}^{\Psi}_{p}}-d_i|\leq \varepsilon_{i}$
and $|\frac{\mathcal{I}^{\Psi}_{i,k}}{\mathcal{I}^{\Psi}_{i,k}}-d|\leq \varepsilon_{i+1}$ for $k\geq \mathcal{N}_{i+1}$.
It follows from Stolz-Ces\`{a}ro theorem that
\begin{equation}\label{27}
\liminf_{r\to 0}\frac{\log(\mu_{\omega}(B(x,r)))}{\log r}\geq d.
\end{equation}

It remains to prove that $\limsup_{r\to 0}\frac{\log \mu_{\omega}(B(x,r))}{\log r}\leq d$. This is  easier than for the $\liminf$ since we just have to choose the smallest $i$ and then the smallest $k=k_{i+1}$ with ${\mathcal{N}_{i+1}}\leq k \leq N_{i+1}$,  such that  there exists $w\in R_i(d_1,d_2,\cdots,d_i)$ and $v\in \mathcal V(\sigma^{m_{i+1}}\omega,i+1,q_{i+1},k)$ for which  $x\in U_{\omega}^{w\ast v}$ and $|U_{\omega}^{w\ast v}|\leq r$.
Then $$\mu_{\omega}(B(x,r))\geq \mu_{\omega}(U_{\omega}^{w\ast v})\geq \exp((\sum_{p=1}^{i}\mathcal{I}^{\Phi}_{p})+ \mathcal{I}^{\Phi}_{i+1,k}-2(m_{i+1}+n^{i+1}_{k})
\varepsilon^2_{i}).$$

\medskip

If $k>\mathcal{N}_{i+1}$, then $\widetilde v$, the father  of $v$, belongs to $ \mathcal V(\sigma^{m_{i+1}}\omega,i+1,d_{i+1},k_{i+1}-1)$ and $x\in U_{\omega}^{w\ast \widetilde{v}}$. We have  $|U_{\omega}^{w\ast \widetilde{v}}|\geq r$, so
\begin{eqnarray*}
	\log r&\leq& \log |U_{\omega}^{w\ast \widetilde{v}}|\leq (\sum_{p=1}^{i}\mathcal{I}^{\Psi}_{p})+\mathcal{I}^{\Psi}_{i+1,k-1}+2(m_{i+1}+n^{i+1}_{k})
	\varepsilon^2_{i}\\
	&\leq& (\sum_{p=1}^{i}\mathcal{I}^{\Psi}_{p})+\mathcal{I}^{\Psi}_{i+1,k}+3(m_{i+1}+n^{i+1}_{k})
	\varepsilon^2_{i}.
\end{eqnarray*}

If $k=\mathcal{N}_{i+1}$, then there exists $w'\in R_{i-1}(d_1,d_2,\cdots,d_{i-1})$, and
$v'\in \mathcal{V}(\sigma^{m_{i}}\omega,i,d_{i},N_{i})$ with
$x\in U_{\omega}^{w'\ast v'}$ and $|U_{\omega}^{w'\ast v'}|\geq r$, so
\begin{eqnarray*}
	\log r&\leq& \log |U_{\omega}^{w'\ast v'}|\leq (\sum_{p=1}^{i}\mathcal{I}^{\Psi}_{p})+2m_{i+1}
	\varepsilon^2_{i}\\
	&\leq& (\sum_{p=1}^{i}\mathcal{I}^{\Psi}_{p})+\mathcal{I}^{\Psi}_{i+1,k}+3(m_{i+1}+n^{i+1}_{k})
	\varepsilon^2_{i}.
\end{eqnarray*}
In both cases we have 
\begin{equation*}
\log r\leq (\sum_{p=1}^{i}\mathcal{I}^{\Psi}_{p})+\mathcal{I}^{\Psi}_{i+1,k}+3(m_{i+1}+n^{i+1}_{k})
\varepsilon^2_{i}.
\end{equation*}
Finally we have 
\begin{equation}\label{28}
\frac{\log(\mu_{\omega}(B(x,r)))}{\log r}\leq \frac{(\sum_{p=1}^{i}\mathcal{I}^{\Phi}_{p})+ \mathcal{I}^{\Phi}_{i+1,k}-2(m_{i+1}+n^{i+1}_{k})
	\varepsilon^2_{i}}{(\sum_{p=1}^{i}\mathcal{I}^{\Psi}_{p})+\mathcal{I}^{\Psi}_{i+1,k}+3(m_{i+1}+n^{i+1}_{k})\varepsilon^2_{i}},
\end{equation}
which yields  
$$\limsup_{r\to 0}\frac{\log(\mu_{\omega}(B(x,r)))}{\log r}\leq d.$$
The inclusion $K(\omega,\{d_i\}_{i\geq 1})\subset E(\mu_\omega,d)$ is established.

Next we estimate $\dim_H K(\omega,\{d_i\}_{i\geq 1})$ from below.
For any  $v=w_0\ast v'\in R_1(d_1)$, with  $\  v'\in \mathcal{V}(\sigma^{{m_1}}\omega,1,q_1)$,  define
\begin{equation}\label{}
\eta_{\omega}(U_{\omega}^{w_0\ast v'}):=\frac{\mu^{\Lambda_{1,q_1}}_{\sigma^{m_1}\omega}(U_{\sigma^{m_1}\omega}^{v'})}{\sum_{v''\in  \mathcal{V}(\sigma^{m_1}\omega,1,q_1)}\mu^{\Lambda_{1,q}}_{\sigma^{m_1}\omega}(U_{\sigma^{m_1}\omega}^{v''})}.
\end{equation}
Then inductively, for any  $w\in R_i(d_1,d_2,\cdots,d_i)$ and $ v\in \mathcal{V}(\sigma^{{m_{i+1}}}\omega,i+1,q_{i+1})$, so that $w\ast v \in R_{i+1}(d_1,d_2,\cdots,d_i,d_{i+1})$,  define:
\begin{equation}\label{}
\eta_{\omega}(U_{\omega}^{w\ast v}):=\eta_{\omega}(U_{\omega}^w) \frac{\mu^{\Lambda_{{i+1},q_{i+1}}}_{\sigma^{m_{i+1}}\omega}(U_{\sigma^{m_{i+1}}\omega}^v)}{\sum_{v'\in \mathcal V_{i+1}}
	\mu^{\Lambda_{{i+1},q_{i+1}}}_{\sigma^{m_{i+1}}\omega}(U_{\sigma^{m_{i+1}}\omega}^{v'})}.
\end{equation}
We can extend $\eta_{\omega}$ in a unique way to a probability measure on the  $\sigma$-algebra  generated by $\cup_{i\geq 1}\{ U_{\omega}^v:v\in R_{i}(d_1,d_2,\cdots, d_{i})\}$. This measure is supported on  $K(\omega,\{d_i\}_{i\geq 1})$. 

Since for each $i\ge 1$ we have $\sum_{v'\in \mathcal V_{i}}
\mu^{\Lambda_{{i},q_{i}}}_{\sigma^{m_{i}}\omega}(U_{\sigma^{m_{i}}\omega}^{v'})\ge \mu^{\Lambda_{{i},q_{i}}}_{\sigma^{m_{i}}\omega}(E_{i,q_i})\ge 1-\epsilon_i$ and $\prod_{i=1}^{\infty}(1-\epsilon_i)\geq \frac{1}{2}$, using the same approach as the proof of \eqref{27}, we can get that  for any $x\in K(\omega,\{d_i\}_{i\geq 1})$,
\begin{equation}\label{6.18}
\liminf_{r\to 0}\frac{\log(\eta_{\omega}(B(x,r)))}{\log r} \geq  \liminf_{i\to \infty}T^*(d_i)
= T^*(d).
\end{equation}
This yields  $\dim_H(E(\mu_{\omega},d))\geq T^\ast(d)$.

\begin{remark}\label{control local}In fact, if in the construction the  sequence $(m_i)_{i\in\N}$ is replaced by another one growing faster (with the effect to modify $(K(\omega,\{d_i\}_{i\in\N}),\eta_\omega)$) (for example $m_{i-1}\leq m_i\varepsilon^3_{i}$),  for all $x\in K(\omega,\{d_i\}_{i\in\N})$, 
	\begin{equation}\label{32}
	\liminf_{r\to 0}\frac{\log\eta_\omega(B(x,r))}{\log r}=\liminf_{i\to\infty}T^*(d_i).
	\end{equation}
	\begin{equation}\label{33}
	\limsup_{r\to 0}\frac{\log\eta_\omega(B(x,r))}{\log r}=\limsup_{i\to\infty}T^*(d_i).
	\end{equation}
	This property will be used in the next subsection.
\end{remark}
\begin{proof}[Proof of \eqref{32} and \eqref{33}] We have 
	\begin{eqnarray*}
		&&\liminf_{r\to 0}\frac{\log\eta_\omega(B(x,r))}{\log r}\\
		&\geq&\liminf_{i\to\infty}\frac{(\sum_{p=1}^{i}\mathcal{I}^{\Lambda_{{p},q_{p}}}_{p})+ \mathcal{I}^{\Lambda_{{i+1},q_{i+1}}}_{i+1,k}+3(m_{i+1}+n^{i+1}_{k})
			\varepsilon^2_{i}}{(\sum_{p=1}^{i}\mathcal{I}^{\Psi}_{p})+\mathcal{I}^{\Psi}_{i+1,k+1}-3(m_{i+1}+n^{i+1}_{k})
			\varepsilon^2_{i}}\geq\liminf_{i\to \infty}T^{*}(d_i),
	\end{eqnarray*}
	where for the first inequality we have used the same estimates as to get \eqref{6.13}. Also, 		
	\begin{eqnarray*}
		&&\liminf_{r\to 0}\frac{\log\eta_\omega(B(x,r))}{\log r}\\
		&\leq&\liminf_{i\to \infty}\frac{\log\eta_\omega(B(x,|U_{\omega}^v|))}{\log |U_{\omega}^v|} \text{ (here } v\in R_i(d_1,\cdots,d_i) \text{ and } x\in U_{\omega}^v\text{)}\\
		&\leq&\liminf_{i\to\infty}\frac{(\sum_{p=1}^{i}\mathcal{I}^{\Lambda_{{p},q_{p}}}_{p})+3(m_{i}+n^{i}_{N_i})
			\varepsilon^2_{i}}{(\sum_{p=1}^{i}\mathcal{I}^{\Psi}_{p})-3(m_{i}+n^{i}_{N_i})
			\varepsilon^2_{i}}
		=\liminf_{i\to \infty}T^{*}(d_i),
	\end{eqnarray*}
	if we assume that $m_{i-1}\leq m_i\varepsilon^3_{i}$. Thus equation \eqref{32} is established. To get \eqref{33}, observe that 
	\begin{eqnarray*}
		&&\limsup_{r\to 0}\frac{\log\eta_\omega(B(x,r))}{\log r}\\
		&\geq&\limsup_{i\to\infty}\frac{\log\eta_\omega(B(x,\frac{1}{2}|U_{\omega}^{v}|\exp(-(m_i+n_{N_i}^{i})\epsilon_{i-1}^2)))}{\log |U_{\omega}^{v}|-(m_i+n_{N_i}^{i})\epsilon_{i-1}^2-\log 2}\\
		&&\text{ (here } v\in R_i(d_1,\cdots,d_i) \text{ and } x\in U_{\omega}^v\text{)}\\
		&\geq&\limsup_{i\to\infty}\frac{(\sum_{p=1}^{i}\mathcal{I}^{\Lambda_{{p},q_{p}}}_{p})+3(m_{i}+n^{i}_{k})
			\varepsilon^2_{i-1}}{(\sum_{p=1}^{i}\mathcal{I}^{\Psi}_{p})-3(m_{i}+n^{i}_{k})
			\varepsilon^2_{i-1}}
		=\limsup_{i\to \infty}T^{*}(d_i),
	\end{eqnarray*}
	if we assume that $m_{i-1}\leq m_i\varepsilon^3_{i}$. On the other hand,
	\begin{eqnarray*}
		&&\limsup_{r\to 0}\frac{\log\eta_\omega(B(x,r))}{\log r}\\
		&\leq&\limsup_{i\to \infty}\frac{(\sum_{p=1}^{i}\mathcal{I}^{\Lambda_{{p},q_{p}}}_{p})+ \mathcal{I}^{\Lambda_{{i+1},q_{i+1}}}_{i+1,k}+3(m_{i+1}+n^{i+1}_{k})
			\varepsilon^2_{i}}{(\sum_{p=1}^{i}\mathcal{I}^{\Psi}_{p})+\mathcal{I}^{\Psi}_{i+1,k+1}-3(m_{i+1}+n^{i+1}_{k})
			\varepsilon^2_{i}}
		\leq\limsup_{i\to \infty}T^{*}(d_i).
	\end{eqnarray*}
\end{proof}

%\begin{rem}
%	From the proof of theorem \ref{multifractal initial} we can directly get that
%	$$\eta_{\omega}(\{x:\lim_{n\to \infty}\frac{S_n\varphi(\omega,x)}{S_n\psi(\omega,x)}=d\})=1,$$
%	and then
%	$$\dim_H(\{x\in X_{\omega}:\lim_{n\to \infty}\frac{S_n\varphi(\omega,x)}{S_n\psi(\omega,x)}=d\})\geq T^\ast(d).$$
%\end{rem}
\subsection{Proofs of theorem~\ref{multifractal initial}(3)(4) and theorem~\ref{multifractal initial2}}
We first notice that theorem~\ref{multifractal initial}(4) follows from theorem~\ref{multifractal initial}(3) and  proposition 1.3 (2) in \cite{Barral2015}.

\begin{proof}[Proof of theorem~\ref{multifractal initial}(3)]
	
	We first deal with the lower bounds for the dimensions.
	
	At first, for the  Hausdorff dimension let us take two sequences $(d_i)_{i\ge 1}$ and $(d'_i)_{i\ge 1}$ in $\prod_{i\ge 1}D_i$ such that $\lim_{i\to \infty} d_i=d$ and $\lim_{i\to\infty} d'_i=d'$, with the properties:
	$$\lim_{i\to \infty} T^{*}(d_i)= T^{*}(d),\ \lim_{i\to \infty} T^{*}(d'_i)= T^{*}(d').$$
	Set  $\tilde{d}_{2i}=d_i$ and $\tilde{d}_{2i+1}=d'_i$.
	
	We can use the same construction as in the previous subsection and get a set $K(\omega,\{\tilde{d}_i\}_{i\geq 1})\subset E(\mu_{\omega},d,d')$, as well as a probability  measure $\eta_\omega$ supported on $K(\omega,\{\tilde{d}_i\}_{i\geq 1})$.
	
	According to the  proof of remark~\ref{control local}, one can also ensure that  that for all $x\in K(\omega,\{\tilde{d}_i\}_{i\geq 1})$ one has $\liminf_{r\to 0}\frac{\log(\eta_{\omega}(B(x,r)))}{\log r}= \inf\{T^\ast(d),T^\ast(d')\}$. Consequently we get $\dim_H E(\mu_{\omega},d,d')\geq\inf\{T^*(d),T^{*}(d')\}$ by definition of $\dim_H(\eta_{\omega})$. 
	
	For the packing dimension, we choose three sequences $(d_i)_{i\ge 1}$, $(d'_i)_{i\ge 1}$  and $(d''_i)_{i\ge 1}$ in $\prod_{i\ge 1}D_i$ in such a way that $\lim_{i\to \infty} d_i=d$, $\lim_{i\to\infty} d'_i=d'$, $\lim_{i\to\infty} d''_i=d''$, $\lim_{i\to \infty} T^{*}(d_i)= T^{*}(d),\ \lim_{i\to \infty} T^{*}(d'_i)= T^{*}(d'),$
	and
	$\ \lim_{i\to \infty} T^{*}(d''_i)= T^{*}(d'')=\sup\{T^*(\beta):\beta\in [d,d']\}.$ Then we take $\tilde{d}_{3i}=d_i$, $\tilde{d}_{3i+1}=d''_i$ and $\tilde{d}_{3i+2}=d'_i$.
	
	Here again, we get $K(\omega,\{\tilde{d}_i\}_{i\geq 1})\subset E(\mu_{\omega},d,d')$ and the measure $\eta_\omega$, which satisfies $\limsup_{r\to 0}\frac{\log(\eta_{\omega}(B(x,r)))}{\log r}= \sup\{T^*(d),T^{*}(d'),T^{*}(d'')\}=T^{*}(d'')$. Consequently we get $$\dim_P E(\mu_{\omega},d,d')\geq T^{*}(d'')$$ by definition of $\dim_P(\eta_{\omega})$. 
	
	\medskip		
	
	The upper bound of the dimensions   directly come from   proposition 1.3(1) and (1.4) and (1.5) in~\cite{Barral2015}. 
\end{proof}

\begin{proof}[Proof of theorem~\ref{multifractal initial2}] The properties $\mathcal H^g(E(\mu_\omega,d))=0$ if $\limsup_{r\to 0^+}\frac{\log(g(r))}{\log(r)}>T'(d)$ and $\mathcal P^g(E(\mu_\omega,d))=0$ if $\liminf_{r\to 0^+}\frac{\log(g(r))}{\log(r)}>T'(d)$ follow from standard estimates. 
	
	Suppose $d\in [T'(+\infty),T'(-\infty)]$ and $T^*(d)<\max (T^*)$. The estimates used in the proof of \eqref{27} and the proof of remark~\ref{control local} can be used to get a positive sequence $(\varepsilon_i')_{i\in\N}$ decreasing to 0 and a constant $C'>0$ such that, independently on $\{d_i\}_{i\in\N}\in \prod_{i\ge 1} D_i$, for all $x\in K(\omega,\{d_i\}_{i\in\N})$, for $i$ large enough, if $ \exp(-m_{i+1}c_{\Psi}/2)< r$, then 
	\begin{equation}\label{vgc}
	\eta_\omega(B(x,r))\le C' r^{\min \{T^*(d_j)-\varepsilon'_j:1\le j\le i\}}.
	\end{equation} 
	Since $T^*(d)<t_0$, we can find $\{d_i\}_{i\in\mathbb{N}}\in\prod_{i\in\mathbb{N}}D_i$ such that $d_i\to d$ as $i\to \infty$,  $T^*(d_i)-\varepsilon'_i\geq T^*(d)+\varepsilon'_i$  for $i$ large enough,  $T^*(d_i)\to T^*(d)$ as $i\to \infty$, and $T^*(d_i)-\varepsilon'_i$ is ultimately non increasing. 
	
	For any gauge function $g$ such that $\limsup_{r\to 0}\frac{\log g(r)}{\log r}\leq T^*(d)$, there exists a positive sequence $\{\upsilon_r\}_{r>0}$ such that  both $\upsilon_r$ and $r^{\upsilon_r}$ decrease to~0 as $r$ decreases to~0 and $$ g(r)\geq r^{T^*(d)+\upsilon_r}\quad (r\le 1).$$
	Due to \eqref{vgc} for $i$ large enough, for any $r$ such that $ \exp(-m_{i+1}c_{\Psi}/2)\leq r\le \exp(-m_{i}c_{\Psi}/2)$,  for any $x\in K(\omega,\{d_i\}_{i\geq 1})$, 
	$$\eta_\omega(B(x,r))\leq C' r^{\min\{T^*(d_j)-\varepsilon'_j:1\le j\le i\}}\leq C' r^{T^*(d_i)-\varepsilon'_i}\le C'r^{T^*(d)+\varepsilon'_i}.$$
	
	Notice that  $g(r)r^{\upsilon_r}\geq r^{T^*(d)+2\upsilon_r}$. So, if we can impose  $2\upsilon_r\leq \varepsilon'_i$, we will have 
	$$\eta_\omega(B(x,r))\leq C' g(r)r^{\upsilon_r}$$
	hence 
	$$g(r)\geq {C'}^{-1}\eta_{\omega}(B(x,r))r^{-\upsilon_r}.$$ 
	Then, for any positive real number $\delta>0$, this will yield  $\mathcal{H}_{\delta}^{g}(K(\omega,\{d_i\}_{i\geq 1}))\geq {\delta}^{-\upsilon_{\delta}}$, and  letting  $\delta\to 0$,   $\mathcal{H}^{g}(K(\omega,\{d_i\}_{i\geq 1}))=+\infty$ so $\mathcal{H}^{g}(E(\mu_{\omega},d))=+\infty$.
	
	Now,  if we choose $m_i$ large enough so  that 
	$$\upsilon_{\exp(-m_{i}c_{\Psi}/2)}\leq \varepsilon_i'/2,$$
	then for $ \exp(-m_{i+1}c_{\Psi}/2)\leq r\le \exp(-m_{i}c_{\Psi}/2)$, we have $2\upsilon_r\leq \varepsilon'_i$ since $\upsilon_r\leq\upsilon_{\exp(-m_{i}c_{\Psi}/2)}$.
	
	\medskip	
	
	Finally, suppose that $g$ is a gauge function such that $\liminf_{r\to 0}\frac{\log g(r)}{\log r}\leq T^*(d)$. There exist $\{r_j\}_{i\in \mathbb{N}}\in (0,1)^\N$, and $\{\upsilon_{r_j}\}_{j\in\N}\in(0,\infty)^\N$ such that  $\upsilon_{r_j}\in(0,1]$ and $r_j^{\upsilon_j}$ decrease to 0 as $j$ tends to $\infty$,  and $$ g(2r_j)\geq r_j^{T^*(d)+\upsilon_{r_j}}.$$  
	
	Using the same approach as above, we can choose $(d_i)_{\ge 1}\in \prod_{i\ge 1}D_i$ such that $\lim_{i\to \infty} d_i=d$, $T^*(d_i)$ converges slowly to  $T^*(d)$ from above, and in the construction of  $(K(\omega,\{d_i\}_{i\geq 1}),\eta_\omega)$,   $m_i$ tends fast enough to $\infty$ so that,  for some $j_0\in\N$, for all $j\ge j_0$, for any $x\in K(\omega,\{d_i\}_{i\geq 1}$,
	$$\eta_{\omega}(B(x,r_j))\leq  C'(2r_j)^{T^*(d)+2\upsilon_{r_j}}.
	$$
	Now, let $ A\subset K(\omega,\{d_i\}_{i\geq 1})$ be of positive $\eta_\omega$-measure. 
	For any given $\delta>0$, take $j'_0\ge j_0$ such that $r_{j'_0}\leq \delta$ consider the following family of closed balls
	$$\mathcal{B}_k=\{B(x,r_{j}): x\in A,j\ge j'_0 \},$$
	which is a covering of $A$. Due to Besicovitch covering theorem, we can extract an at most countable subfamily of pairwise disjoint balls $\{B(x_i, \rho_i)\}_{i\in I}$ such that $\eta_\omega(\bigcup_{i\in I}B_i)>0$.  This family is a $\delta$-packing of $A$, and ($\mathcal{P}_{0,\delta}^{g}$ stands for the prepacking measure associated with $g$)
	\begin{eqnarray*}
		\mathcal{P}_{0,\delta}^{g}(A)&\geq & \sum_{i}g(B(x_i, \rho_i))\geq \sum_{i}\rho_i^{T^*(d)+\upsilon_{\rho_i}}\\
		&\geq& \sum_{s} \rho_i^{-\upsilon_{\rho_i}} \eta_{\omega}(B(x_i,\rho_i))
		\geq \rho_{j'_0}^{-\upsilon_{\rho_{j'_0}}}\eta_{\omega}(A).
	\end{eqnarray*}
	As $j'_0\to \infty$ when $\delta\to 0$, we can conclude that $P^{g}_{0}(A)=+\infty$. Then, since any at most countable covering of $K(\omega,\{d_i\}_{i\geq 1})$ must contain a set $A$ of positive $\eta_\omega$-measure, we finally  get $\mathcal{P}^g(K(\omega,\{d_i\}_{i\geq 1}))=+\infty$, so $\mathcal{P}^g(E(\mu_{\omega},d))=+\infty$.	%\end{enumerate}
\end{proof}

%For acknowledgements section, please don't number the section, please begin it with \section*{Acknowledgements}
%\section*{Acknowledgments}The author would like to express his heartfelt gratitude to his supervisors: Professor Julien Barral and Zhying Wen

% You may incorporate your references as follows in your main tex file.
% Using BibTex is not recommended but can be handled.

\medskip
% The data information below will be filled by AIMS editorial staff
Received xxxx 20xx; revised xxxx 20xx.
\medskip

\end{document}